\documentclass[12pt]{article}
\usepackage{amsmath,amsfonts,amssymb,amsthm}
\usepackage{graphics}
\usepackage{epsfig}
\usepackage{xcolor}

\usepackage[nottoc]{tocbibind}

\newcommand\numberthis{\addtocounter{equation}{1}\tag{\theequation}}

\newcommand{\ignore}[1]{}

\newtheorem{proposition}{Proposition}
\newtheorem{theorem}{Theorem}
\newtheorem{remark}{Remark}
\newtheorem{lemma}{Lemma}

\newcommand{\R}{\mathbb{R}}
\newcommand{\Z}{\mathbb{Z}}

\newcommand{\F}{\mathcal{F'}}

\newcommand{\ES}{\mathcal{S}}
\def\Ra{{\rm Ra}}

\xdefinecolor{darkgreen}{rgb}{0,0.55,0}

\author{Antoine Choffrut, Camilla Nobili and Felix Otto}
\title{A maximal regularity estimate for the non-stationary Stokes equation in 
the strip}

\begin{document}

\maketitle

 \begin{abstract}
 In a $d-$dimensional strip with $d\geq 2$, we study the non-stationary Stokes
 equation with no-slip boundary condition in the lower and upper plates and 
 periodic boundary condition in the horizontal directions.
 In this paper we establish a new maximal regularity estimate in the real interpolation norm
 \begin{equation*}
       ||f||_{(0,1)}=\inf_{f=f_0+f_1}\left\{\left\langle\sup_{0<z<1} |f_0|\right\rangle+
       \left\langle\int_0^{1} |f_1| \frac{dz}{(1-z)z}\right\rangle\right\}\,,
      \end{equation*}
      where the brackets $\langle\cdot\rangle$ denotes the horizontal-space and time average.
  The norms involved in the definition of $\|\cdot\|_{(0,1)}$ are critical for two reasons: 
  the exponents are borderline for the Calder\'on-Zygmund theory and the weight $1/z$ 
  just fails to be Muckenhoupt.  
  Therefore, the estimate is only true under horizontal 
 bandedness condition, (i.\ e.\ a restriction to a packet of wave numbers in Fourier space).
 The motivation to express the maximal 
 regularity in such a norm comes from an application to the Rayleigh-B\'enard problem
 (see \cite{CNO-part1}).

\smallskip


\end{abstract}

\newpage
\tableofcontents

\newpage

 \section{Introduction}           
 In the $d-$dimensional strip $[0,L)^{d-1}\times[0,1]$,  $d\geq 2$,  we consider the non-stationary Stokes equation for the vector field $u(x',z,t)$ and the
  scalar field $p(x',z,t)$
 \begin{equation}\label{STOKES-STRIP}
	\left\{\begin{array}{rclc}
	 \partial_t u-\Delta u+\nabla p&=&f  \qquad & {\rm for } \quad 0<z<1 \,,\\
	  \nabla\cdot u&=& 0                 \qquad & {\rm for } \quad 0<z<1\,,\\
	   u &=& 0                           \qquad & {\rm for } \quad z\in\{0,1\}\,,\\
	   u &=& 0                           \qquad & {\rm for } \quad t=0\,,\\
         \end{array}\right.        
  \end{equation}	
  where $x'\in[0,L)^{d-1}$ and $z\in [0,1]$ indicate the spatial variables and $t\in \R^+$ denotes the time variable.
  In what follows it is important to distinguish the horizontal component $u'\in \R^{d-1}$
  and the vertical component $u^z\in \R$ of the vector field $u$.
  
  Motivated by an application to the Rayleigh-B\'enard convection problem (see \cite{CNO-part1}), in this paper we establish the following maximal regularity estimate :

  \begin{theorem}[Maximal regularity in the strip]\label{th1}\ \\ 
      There exists $R_0\in(0,\infty)$ depending only on $d$ and $L$ such that the following holds.     
      Let $u,p,f$ satisfy  the equation (\ref{STOKES-STRIP}).
      Assume $f$ is horizontally  band-limited , i.e
      \begin{equation}\label{BANDLIM}
      \F f(k',z,t)=0 \mbox{ unless } 1\leq R|k'|\leq 4  \mbox{ where } R<R_0.
      \end{equation}
      Then,
      \begin{equation}\label{MRE}
      ||(\partial_t -\partial_z^2)u'||_{(0,1)}+||\nabla'\nabla u'||_{(0,1)}+||\partial_t u^z||_{(0,1)}+||\nabla^2 u^z||_{(0,1)}+
      ||\nabla p||_{(0,1)}\lesssim||f||_{(0,1)},
      \end{equation}
      where $||\cdot||_{(0,1)}$ denotes the norm
      \begin{equation}\label{NORM-STRIP}
       ||f||_{(0,1)}:=||f||_{(R,(0,1))}=\inf_{f=f_0+f_1}\left\{\left\langle\sup_{0<z<1} |f_0|\right\rangle+
       \left\langle\int_0^{1} |f_1| \frac{dz}{(1-z)z}\right\rangle\right\}\,,
      \end{equation}
      where $f_0$ and $f_1$ satisfy the bandedness assumption (\ref{BANDLIM})\,.
  \end{theorem}

  In the Theorem above, $\F$ denotes the horizontal Fourier transform, $k'$ the conjugate variable of $x'$
  and the brackets $\langle\cdot\rangle$ stand for long-time and horizontal-space average. See Section \ref{notations} for notations.
  
  The Theorem as stated above is used in this form in \cite{CNO-part1}. Alternatively, the theorem can 
  be stated with the brackets $\langle\cdot\rangle$ denoting the integration in $t>0$, see Remark \ref{TIME}
  at the beginning of Section \ref{tec}.
 The maximal regularity in the strip is expressed in terms of the  \textit{interpolation} between the norms of
 $L^1\left(dtdx'\frac{1}{z(1-z)}dz\right)$ and  $L^\infty_z(L^1_{t,x'})$, which are both borderline
 for the Calder\'on-Zygmund estimates. We notice that the norm of $L^1\left(dtdx'\frac{1}{z(1-z)}dz\right)$ 
 is critical both because of the exponent and the weight $\frac{1}{z(1-z)}$ are borderline,
 therefore, estimate (\ref{MRE}) is only true under \textit{bandedness assumptions}
  (i.\ e.\ a restriction to a packet of wave numbers in Fourier space).
  We observe that only bandedness in the {\it horizontal} variable $x'$ is assumed and 
 this is extremely convenient since the horizontal Fourier transform (or rather, series), 
 with help of which bandedness is expressed, is compatible with the lateral periodic boundary conditions.
 
 We notice that in the maximal regularity theory the no-slip boundary condition
 is a nuisance : As opposed to the no-stress boundary condition in the half space, the no-slip boundary condition
 does not allow for an extension by reflection to the whole space, and thereby the use of simple kernels 
 or Fourier methods also in the normal variable. 
 The difficulty coming from the the no-slip boundary condition in the non-stationary Stokes equations when 
 deriving maximal regularity estimates is of course well-known; many techniques have been developed to derive
 Calder\'on-Zygmund estimates despite this difficulty.
 In the {\it{half space}} Solonnikov in \cite{solonnikov1964} has constructed a solution formula 
 for (\ref{STOKES-STRIP}) with zero initial data via the Oseen an Green tensors. 
 An easier and more compact representation of the solution to the problem (\ref{STOKES-STRIP}) with zero forcing term and non-zero initial value was 
 later given by Ukai in \cite{Ukai} by using a different method. Indeed he could  write an explicit formula
 of the solution operator as a composition of Riesz' operators and solutions operator for the heat and Laplace's
 equation. This formula is an effective tool to get $L^p-L^q$ ($1<q,p<\infty$) estimates for the solution and its derivatives.
 In the case of {\it exterior domains}, Maremonti and Solonnikov \cite{Mar-Sol} derive
 $L^p-L^q$ ($1<q,p<\infty$) estimates for (\ref{STOKES-STRIP}), going through estimates 
 for the extended solution in the half space and in the whole space.
 In particular in the half space they propose a decomposition of (\ref{STOKES-STRIP})
 with non-zero divergence equation.
 The book of Galdi \cite{galdi} provides with a complete treatment of the classical theory and
 results on the non-stationary Stokes equations and Navier-Stokes equations.

  In \cite{CNO-part1} the authors make substantial use of the estimate (\ref{MRE}) in Theorem \ref{th1} to get bounds on the 
the Nusselt number, which is the natural measure of the enhancement of upward heat flux for the Rayleigh-B\'enard convection. There, the quantity of interest 
 is the second vertical derivative $\partial_z^2$ of the vertical velocity component $u^z=u\cdot e_z$.
 The motivation for expressing the maximal regularity in the {\it borderline} spaces
 $L^1\left(dtdx'\frac{1}{z(1-z)}dz\right)$ and $L^\infty_z(L^1_{t,x'})$ comes from the nature 
 of the right-hand-side $f=\Ra Te_z-\frac{1}{Pr}(u\cdot \nabla) u$ in the problem studied in \cite{CNO-part1}. Indeed, thanks to the no-slip boundary conditions,
 the convective nonlinearity is well controlled in the $L^1\left(dtdx'\frac{1}{z(1-z)}dz\right)$-norm,
  hence, a maximal regularity theory for the non-stationary Stokes equations with respect to this norm
 is required. 
 The $L^\infty_z(L^1_{t,x'})-$ norm arises for two unrelated reasons: It is needed to estimate the buoyancy term $Te_z$ 
driving the Navier-Stokes equations and it is the natural partner of $L^1\left(dtdx'\frac{1}{z(1-z)}dz\right)$
in the maximal regularity estimate.

 Aside from their application to the Rayleigh B\'enard convection
 all the estimates in Theorem \ref{th1} might have 
 an independent interest since they
 show the full extent of what one can 
 obtain under the horizontal bandedness assumption only.

\section{Maximal regularity in the strip}

\subsection{From the strip to the half space}
Let us consider the non-stationary Stokes equations
\begin{equation*}
  \left\{\begin{array}{rclc}
     \partial_t u-\Delta u+\nabla p &=& f \qquad & {\rm for } \quad 0<z<1 \,,\\
        \nabla\cdot u &=& 0 \qquad & {\rm for }  \quad 0<z<1\,,\\
         u &=& 0  \qquad & {\rm for } \quad z\in\{0,1\} \,,\\
         u &=& 0  \qquad  & {\rm for } \quad t=0 \,.\\
         \end{array}\right.        
 \end{equation*}
In order to prove the maximal regularity 
estimate in the strip we extend the problem 
(\ref{STOKES-STRIP}) in the half space.
By symmetry, it is enough to consider for the moment
the extension to the upper half space.
\newline
Consider the localization $(\tilde u, \tilde p):=(\eta u,\eta p)$ where
 
 \begin{equation}\label{cutoff} \eta(z) \mbox{ is a cut-off function for } \left[0,\frac{1}{2}\right) \mbox{ in } [0,1) \,.\end{equation}
 Extending $(\tilde u, \tilde p)$ by zero they can be viewed as functions in the upper half space.
The couple  $(\tilde u, \tilde p)$ satisfies 

 \begin{equation}\label{UHS}
  \left\{\begin{array}{rclc}
      \partial_t \tilde u-\Delta \tilde u+\nabla \tilde p &=& \tilde f \qquad & {\rm for } \quad z>0\,,\\
        \nabla\cdot \tilde u &=& \tilde\rho \qquad & {\rm for } \quad z>0 \,,\\
        \tilde u &=&  0 \qquad & {\rm for } \quad z=0\,,\\
         \tilde u &=& 0  \qquad  & {\rm for } \quad t=0 \,,\\
         \end{array}\right.        
 \end{equation}
where 
\begin{equation}\label{Defi}
\tilde f:=\eta f-2(\partial_z \eta)\partial_z u-(\partial_z^2\eta )u+(\partial_z\eta )pe_z, \qquad \qquad  \tilde\rho:=(\partial_z\eta )u^z\,.
\end{equation}
%

\subsection{Maximal regularity in the upper half space}
In the half space, taking 
advantages from the explicit representation of the solution 
via Green functions, we prove 
the regularity estimates which will be crucial in the proof of
Theorem \ref{th1}.
\begin{proposition}[Maximal regularity in the upper half space]\label{pr1}\ \\
Consider the non-stationary Stokes equations in the upper half-space 
 \begin{equation}\label{STOKES-HALF}
  \left\{\begin{array}{rclc}
      \partial_t u-\Delta u+\nabla p &=& f \qquad & {\rm for } \quad z>0\,,\\
         \nabla\cdot u &=& \rho \qquad & {\rm for } \quad z>0 \,,\\
         u &=&  0 \qquad & {\rm for } \quad z=0\,,\\
         u &=&  0 \qquad & {\rm for } \quad t=0\,.\\
         \end{array}\right.        
 \end{equation}
 Suppose that $f$ and $\rho$ are horizontally band-limited , i.e
\begin{equation}\label{BC1}\F f(k',z,t)=0 \mbox{ unless } 1\leq R|k'|\leq 4  \mbox{ where } R\in(0,\infty)\,,\end{equation}
and 
\begin{equation}\label{BC2}\F \rho(k',z,t)=0 \mbox{ unless } 1\leq R|k'|\leq 4  \mbox{ where } R\in(0,\infty)\,.\end{equation}
Then
\begin{eqnarray*}
 &&||\partial_t u^z||_{(0,\infty)}+||\nabla^2 u^z||_{(0,\infty)}+||\nabla p||_{(0,\infty)}+||(\partial_t -\partial_z^2)u'||_{(0,\infty)}+||\nabla'\nabla u'||_{(0,\infty)}\\
 &\lesssim&||f||_{(0,\infty)}+||(-\Delta')^{-\frac{1}{2}}\partial_t \rho||_{(0,\infty)}+||(-\Delta')^{-\frac{1}{2}}\partial_z^2 \rho ||_{(0,\infty)}+||\nabla \rho||_{(0,\infty)},
\end{eqnarray*}
where $||\cdot||_{(0,\infty)}$ denotes the norm
\begin{equation}\label{NORM-HALF}||f||_{(0,\infty)}:=||f||_{R;(0,\infty)}\inf_{f=f_0+f_1}\left\{\left\langle\sup_{0<z<\infty} |f_0|\right\rangle+\left\langle\int_0^{\infty} |f_1|\frac{dz}{z}\right\rangle\right\}\,,\end{equation}
where $f_0$ and $f_1$ satisfy the bandedness assumption (\ref{BC1}).
\end{proposition}
The first ingredient to establish Proposition \ref{pr1} is a suitable representation 
of the solution operator $(f=(f',f^z),\rho)\rightarrow u=(u',u^z) $
of the Stokes equations with the no-slip  boundary condition.
In the case of no-slip boundary condition the Laplace operator  
has to be factorized as
$\Delta=\partial_z^2+\Delta'=(\partial_z+(-\Delta')^{\frac{1}{2}})(\partial_z-(-\Delta')^{\frac{1}{2}})$.
In this way the solution operator  to the Stokes  equations with the no-slip boundary
 condition (\ref{STOKES-HALF}) can be written as the fourfold composition of solution operators 
to three more elementary boundary value problems: 
\begin{itemize}

\item Backward fractional diffusion equation (\ref{FraBack}):

\begin{equation}\label{FraBack}
  \left\{\begin{array}{rclc}
     (\partial_z-(-\Delta')^{\frac{1}{2}})\phi &=& \nabla\cdot f-(\partial_t-\Delta )\rho \qquad & {\rm for } \quad  z>0 \,,\\
        \phi &\rightarrow& 0  \qquad & {\rm for } \quad z\rightarrow \infty.\\
         \end{array}\right.        
 \end{equation}

\item  Heat equation (\ref{Heat1}):
\begin{equation}\label{Heat1}
  \left\{\begin{array}{rclc}
  (\partial_t-\Delta)v^z&=&(-\Delta')^{\frac{1}{2}}(f^z-\phi)-\nabla'\cdot f'+(\partial_t-\Delta)\rho \qquad & {\rm for } \quad  z>0,\\
  v^z&=& 0  \qquad & {\rm for } \quad  z=0 \,,\\
  v^z&=& 0  \qquad & {\rm for } \quad  t=0 \,.\\
  \end{array}\right. 
\end{equation}

\item  Forward fractional diffusion equation (\ref{FraFor}): 

\begin{equation}\label{FraFor}
  \left\{\begin{array}{rclc}
  (\partial_z+(-\Delta')^{\frac{1}{2}})u^z&=& v^z  \qquad & {\rm for } \quad   z>0\,,\\
  u^z&=&0   \qquad & {\rm for } \quad  z=0\,.\\
  \end{array}\right. 
\end{equation}

\item Heat equation (\ref{Heat2}):

\begin{equation}\label{Heat2}
 \left\{\begin{array}{rclc}
(\partial_t-\Delta)v'&=&(1+\nabla'(-\Delta')^{-1}\nabla'\cdot)f'   \qquad & {\rm for } \quad   z>0\,,\\
v'&=& 0   \qquad & {\rm for } \quad z=0\,,\\
v'&=& 0   \qquad & {\rm for } \quad  t=0\,.\\
  \end{array}\right.
\end{equation}
Finally set
\begin{equation}
\label{horvel}u'=v'-\nabla'(-\Delta')^{-1}(\rho-\partial_z u^z)\,.
\end{equation}

\end{itemize}
In order to prove the validity of the decomposition we need to argue that 
$$(\partial_t-\Delta)u-f \mbox{ is irrotational }\,,$$ which 
reduces to prove that
 \begin{equation*}
  (\partial_t-\Delta)u'-f' \mbox{ is irrotational in } x'
 \end{equation*}
 and
 \begin{equation}\label{SEC}
  \partial_z ((\partial_t-\Delta)u'-f')=\nabla'((\partial_t-\Delta)u^z-f^z)\,.
 \end{equation} 
 Let us consider for simplicity $\rho=0.$
The first statement follows easily from the definition. Indeed by definition (\ref{horvel}) and  equation (\ref{Heat2}), 
 $$ (\partial_t-\Delta)u'-f' =\nabla' ((-\Delta')^{-1}\nabla'\cdot f'+(-\Delta')^{-1}\partial_z u^z).$$
Let us now focus on (\ref{SEC}), which by using (\ref{horvel}) and (\ref{Heat2}) can be rewritten as
  $$\partial_z \nabla'((-\Delta')^{-1}\nabla'\cdot f'+(-\Delta')^{-1}(\partial_t-\Delta)\partial_zu^z)=\nabla'((\partial_t-\Delta)u^z-f^z)\,.$$
Because of the periodic boundary conditions in the horizontal direction, the latter is equivalent to
$$\partial_z (-\Delta')((-\Delta')^{-1}\nabla'\cdot f'+(-\Delta')^{-1}(\partial_t-\Delta)\partial_zu^z)=(-\Delta')((\partial_t-\Delta)u^z-f^z),$$
that, after factorizing  $\Delta=(\partial_z-(-\Delta')^{\frac{1}{2}})(\partial_z+(-\Delta')^{\frac{1}{2}})$, turns into
$$(\partial_z-(-\Delta')^{\frac{1}{2}})(\partial_t-\Delta)(\partial_z+(-\Delta')^{\frac{1}{2}}) u^z=(-\Delta')f^z-\partial_z\nabla'\cdot f'\,.$$
One can easily check that the identity holds true by applying (\ref{FraFor}), (\ref{Heat1}) and (\ref{FraBack}).
The no-slip boundary condition is trivially satisfied, indeed by (\ref{FraFor}) we have $u^z=0$ and $\partial_z u^z=0$. 
The combination of (\ref{horvel}) with $\partial_z u^z=0$ gives $u'=0$.
%
%
%
%
%
%

For each step of the decomposition of the Navier Stokes equations 
we will derive maximal regularity-type estimates. These are summed 
up in the following

\begin{proposition}\label{prop3}\ \\ 
 \begin{enumerate} 
  \item Let $\phi,f,\rho$ satisfy the problem (\ref{FraBack}) and assume
        $f,\rho$ are  horizontally band-limited, i.e 
        $$\F f(k',z,t)=0 \mbox{ unless } 1\leq R|k'|\leq 4$$
        and 
        $$\F \rho(k',z,t)=0 \mbox{ unless } 1\leq R|k'|\leq 4.$$
        
        Then, 
        
        \begin{equation*}
          ||\phi||_{(0,\infty)}\lesssim ||f||_{(0,\infty)}+||(-\Delta')^{-\frac{1}{2}}\partial_t \rho||_{(0,\infty)}+||\nabla \rho||_{(0,\infty)} \,.\label{A}
        \end{equation*}
        
  \item  Let $v^z, f, \phi, \rho$ satisfy the problem (\ref{Heat1}) and assume
          $f,\phi,\rho$ are  horizontally band-limited, i.e 
          $$\F f(k',z,t)=0 \mbox{ unless } 1\leq R|k'|\leq 4\,,$$         
          $$\F \phi(k',z,t)=0 \mbox{ unless } 1\leq R|k'|\leq 4\,$$
          and
          $$\F \rho(k',z,t)=0 \mbox{ unless } 1\leq R|k'|\leq 4\,.$$
          
          Then,
         \begin{equation}\label{B}
         \begin{array}{rclc}
          &&||\nabla v^z||_{(0,\infty)}+||(-\Delta)^{-\frac{1}{2}}(\partial_t-\partial_z^2)v^z||_{(0,\infty)}\\
          &\lesssim& ||f||_{(0,\infty)}+||\phi||_{(0,\infty)}+||(-\Delta')^{-\frac{1}{2}}\partial_t\rho||_{(0,\infty)} \\
          &+&||(-\Delta)^{-\frac{1}{2}}\partial_z^2\rho||_{(0,\infty)}+||\nabla\rho||_{(0,\infty)}\,.       
         \end{array}  
         \end{equation}       
                  
   \item Let $u^z, v^z$ satisfy the problem (\ref{FraFor}) and assume 
         $v^z$ is  horizontally band-limited, i.e 
          $$\F v^z(k',z,t)=0 \mbox{ unless } 1\leq R|k'|\leq 4\,.$$
          Then,   
         \begin{equation}\label{C}
           \begin{array}{rclc}
           &&||\partial_t u^z||_{(0,\infty)}+||\nabla^2u^z||_{(0,\infty)}+||(-\Delta')^{-\frac{1}{2}}\partial_z(\partial_t-\partial_z^2)u^z||_{(0,\infty)}\\
          &\lesssim & ||\nabla v^z||_{(0,\infty)}+||(-\Delta')^{-\frac{1}{2}}(\partial_t-\partial_z^2)v^z||_{(0,\infty)}\,.\\    
           \end{array}              
          \end{equation}
   \item Let $v',f'$, satisfy the problem (\ref{Heat2}) and assume 
          $f'$ is horizontally band-limited, i.e 
          $$\F f(k',z,t)=0 \mbox{ unless } 1\leq R|k'|\leq 4\,.$$          
          Then,
          \begin{equation}\label{D}
           ||\nabla'\nabla v'||_{(0,\infty)}+||(\partial_t-\partial_z^2)v'||_{(0,\infty)}\lesssim ||f'||_{(0,\infty)}\,. 
          \end{equation}
%
 \end{enumerate}

\end{proposition}

\subsection{Proof of Proposition \ref{pr1}}

By an easy application of Proposition \ref{prop3}, we will now prove the maximal regularity estimate on the upper  half space.

\begin{proof}[Proof of Proposition \ref{pr1}]\ \\
From Proposition \ref{prop3} we have the following bound for the  vertical component of the velocity $u$

          \begin{eqnarray*} 
           &&||\partial_t u^z||_{(0,\infty)}+||\nabla^2u^z||_{(0,\infty)}+||(-\Delta')^{-\frac{1}{2}}\partial_z(\partial_t-\partial_z^2)u^z||_{(0,\infty)}\\
          &\stackrel{(\ref{C})}{\lesssim} &||\nabla v^z||_{(0,\infty)}+||(-\Delta')^{-\frac{1}{2}}(\partial_t-\partial_z^2)v^z||_{(0,\infty)}\\
          &\stackrel{(\ref{B})}{\lesssim} &||f||_{(0,\infty)}+||\phi||_{(0,\infty)}+||(-\Delta')^{-\frac{1}{2}}\partial_t\rho||_{(0,\infty)}
          +||(-\Delta)^{-\frac{1}{2}}\partial_z^2\rho||_{(0,\infty)}+||\nabla\rho||_{(0,\infty)} \\
           &\stackrel{(\ref{A})}{\lesssim} & ||f||_{(0,\infty)}+||(-\Delta')^{\frac{1}{2}}\partial_t\rho||_{(0,\infty)}+||(-\Delta)^{-\frac{1}{2}}\partial_z^2\rho||_{(0,\infty)}+||\nabla\rho||_{(0,\infty)} \,.       
         \end{eqnarray*}
 Instead for the horizontal components of the velocity $u'$ we have 
        \begin{eqnarray*}
        &&||(\partial_t-\partial_z^2)u'||_{(0,\infty)}+||\nabla'\nabla u'||_{(0,\infty)}\\
        &\stackrel{(\ref{horvel})}{\lesssim}& ||(\partial_t-\partial_z^2)v'||_{(0,\infty)}+||\nabla'\nabla v'||_{(0,\infty)}\\
        &+&||(-\Delta')^{-\frac{1}{2}}(\partial_t-\partial_z^2)\rho||_{(0,\infty)}+||\nabla\rho||_{(0,\infty)}\\
        &+&||(-\Delta')^{-\frac{1}{2}}\partial_z(\partial_t-\partial_z^2)u^z||_{(0,\infty)}+||\partial_z\nabla u^z||_{(0,\infty)}\\
        &\stackrel{(\ref{B}),(\ref{C}),(\ref{D})}{\lesssim}& ||f||_{(0,\infty)}+||(-\Delta')^{-\frac{1}{2}}\partial_t\rho||_{(0,\infty)}+||(-\Delta)^{-\frac{1}{2}}\partial_z^2\rho||_{(0,\infty)}+||\nabla\rho||_{(0,\infty)}\,.        
      \end{eqnarray*}  
 Summing up we obtain
\begin{equation}\label{ESTI}
\begin{array}{rclc}
&&||\partial_t u^z||_{(0,\infty)}+||\nabla^2u^z||_{(0,\infty)}+||(\partial_t-\partial_z^2)u'||_{(0,\infty)}+||\nabla'\nabla u'||_{(0,\infty)}\\
&\lesssim&||f||_{(0,\infty)}+||(-\Delta')^{-\frac{1}{2}}\partial_t\rho||_{(0,\infty)}+||(-\Delta)^{-\frac{1}{2}}\partial_z^2\rho||_{(0,\infty)}+||\nabla\rho||_{(0,\infty)}\,.
\end{array}
\end{equation}
The bound for the $\nabla p$ follows by equations (\ref{STOKES-HALF}) and applying (\ref{ESTI}).
\end{proof}

\subsection{Proof of Proposition \ref{prop3} }
This section is devoted to the proof of Proposition \ref{prop3}, which rely on a series of Lemmas
(Lemma \ref{lemma1}, Lemma \ref{lemma2} and Lemma \ref{lemma3}) that we state here and prove in Section \ref{tec}.

The following Lemmas contain the basic maximal regularity 
estimates for the three auxiliary problems. These estimates,
together with the bandedness assumption in the form  of (\ref{P}), (\ref{Q})
and (\ref{R}) will be the main ingredients for the proof
of Proposition \ref{prop3}.

\begin{lemma}\label{lemma1}\ \\
  Let $u,f$ satisfy  the problem

 \begin{equation}\label{I}
 \left\{\begin{array}{rclc}
        (\partial_z-(-\Delta')^{\frac{1}{2}})u&=&f  \qquad & {\rm for } \quad z>0\,,\\
        u&\rightarrow&0   \qquad & {\rm for } \quad z\rightarrow\infty\,
      \end{array}\right.
  \end{equation}   
    and assume $f$ to be horizontally band-limited, i.e
   $$\F f(k',z,t)=0 \quad \mbox{ unless }\quad 1\leq R|k'|\leq 4\,.$$
  Then, 
   \begin{equation}\label{1}
    ||\nabla u||_{(0,\infty)}\lesssim||f||_{(0,\infty)}\,.
   \end{equation}
\end{lemma}

\begin{lemma}\label{lemma2}\ \\ 	
   Let $u,f,g=g(x',t)$ satisfy  the problem
    
  \begin{equation}\label{II}
 \left\{\begin{array}{rclc}
        (\partial_z+(-\Delta')^{\frac{1}{2}})u&=&f   \qquad & {\rm for } \quad z>0\,,\\
        u&=&g   \qquad & {\rm for } \quad z=0\,
      \end{array}\right.
  \end{equation}
 and define the constant extension $\tilde g(x',z,t):=g(x',t).$ 
 Assume $f$ and $g$ to be horizontally band-limited, i.e
  $$\F f(k',z,t)=0 \quad\mbox{ unless }\quad 1\leq R|k'|\leq 4\,$$
  and
  $$\F g(k',z,t)=0 \quad\mbox{ unless }\quad 1\leq R|k'|\leq 4\,.$$
  Then
  \begin{equation}\label{3}
   ||\nabla u||_{(0,\infty)}\lesssim||f||_{(0,\infty)}+||\nabla'\tilde{g}||_{(0,\infty)}\,.
     \end{equation}
\end{lemma}	
  
\begin{remark}
 Clearly if $g=0$ in Lemma \ref{lemma2}, then we have
  \begin{equation}\label{2}
  ||\nabla u||_{(0,\infty)}\lesssim||f||_{(0,\infty)}\,.
  \end{equation}
\end{remark}
		  
\begin{lemma}\label{lemma3}\ \\
   Let $u,f$ satisfy  the problem
     \begin{equation}\label{III}
     \left\{\begin{array}{rclc}
	(\partial_t-\Delta) u&=&f \qquad & {\rm for } \quad z>0\,,\\ 
	u&=&0 \qquad & {\rm for } \quad z=0\,,\\
	u&=&0 \qquad & {\rm for } \quad t=0\,\\
      \end{array}\right.
     \end{equation}
     and assume $f$ to be horizontally band-limited, i.e
   $$\F f(k',z,t)=0 \quad\mbox{ unless }\quad 1\leq R|k'|\leq 4\,.$$

Then, 				
 \begin{equation}\label{4}
 ||(\partial_t-\partial_z^2)u||_{(0,\infty)}+||\nabla'\nabla u||_{(0,\infty)}\lesssim ||f||_{(0,\infty)}\,.
 \end{equation}
       
 \end{lemma}       
               
\vspace{2cm}

\begin{proof} [Proof of Proposition \ref{prop3}]\ \\
 
\begin{enumerate}

 \item 

Subtracting the quantity $(\partial_z-(-\Delta')^{\frac{1}{2}})(f^z+\partial_z\rho)$
 from both sides of equation (\ref{FraBack})
 and then multiplying the new equation by $(-\Delta)^{-\frac{1}{2}}$ we get

\begin{eqnarray*}
&&(\partial_z-(-\Delta')^{\frac{1}{2}})(-\Delta')^{-\frac{1}{2}}(\phi-f^z-\partial_z\rho)\\
&=&\nabla'\cdot(-\Delta')^{-\frac{1}{2}} f'+f^z-(-\Delta')^{-\frac{1}{2}}\partial_t\rho+\partial_z \rho-(-\Delta')^{\frac{1}{2}}\rho\,.
\end{eqnarray*}

From the basic estimate (\ref{1}) we obtain

\begin{eqnarray*}
&&||\nabla'(-\Delta')^{-\frac{1}{2}}(\phi-f^z-\partial_z\rho)||_{(0,\infty)}\lesssim ||\nabla'\cdot(-\Delta')^{-\frac{1}{2}} f'||_{(0,\infty)}\\
&+&||f^z||_{(0,\infty)}+||(-\Delta')^{-\frac{1}{2}}\partial_t\rho||_{(0,\infty)}+||\partial_z \rho||_{(0,\infty)}+||(-\Delta')^{\frac{1}{2}}\rho||_{(0,\infty)}\,.
\end{eqnarray*}

Thanks to the bandedness assumption in the form of  (\ref{P})  and (\ref{Q}) we have
\begin{eqnarray*}
&&||\phi-f^z-\partial_z\rho||_{(0,\infty)}\\
&\lesssim& || f'||_{(0,\infty)}+||f^z||_{(0,\infty)}+||(-\Delta')^{-\frac{1}{2}}\partial_t\rho||_{(0,\infty)}+||\partial_z \rho||_{(0,\infty)}+||\nabla'\rho||_{(0,\infty)}
\end{eqnarray*}
and from this we obtain easily the desired estimate (\ref{A}).
\item 

After multiplying the equation  (\ref{Heat1}) by $(-\Delta')^{-\frac{1}{2}}$, 
the application of (\ref{4}) to $(-\Delta')^{-\frac{1}{2}}v^z$ yields

\begin{eqnarray*}
&&||(-\Delta')^{-\frac{1}{2}}(\partial_t-\partial_z^2)v^z||_{(0,\infty)}+||(-\Delta')^{-\frac{1}{2}}\nabla'\nabla v^z||_{(0,\infty)}\\
&\lesssim& ||f^z||_{(0,\infty)}+||\phi||_{(0,\infty)}+||\nabla'\cdot(-\Delta')^{-\frac{1}{2}}f'||_{(0,\infty)}\\
&+&||(-\Delta')^{-\frac{1}{2}}(\partial_t-\partial_z^2)\rho||_{(0,\infty)}+||(-\Delta')^{\frac{1}{2}}\rho||_{(0,\infty)}\,.
\end{eqnarray*}

The estimate (\ref{B}) follows after observing (\ref{Q}) and applying the triangle inequality to the second to last
 term on the right hand side.
 
\item

We need to estimate the the three terms on the right hand side of (\ref{C}) separately.
We start with the term $\nabla^2 u^z$:
 since $||\nabla^2 u^z||_{(0,\infty)}\leq||\nabla'\nabla u^z||_{(0,\infty)}+||\partial_z^2 u^z||_{(0,\infty)}$, 
we tackle the term $\nabla'\nabla u^z$ and $\partial_z^2 u^z$ separately.
First multiply by $\nabla'$ the equation (\ref{FraFor}).  An application
 of the estimate (\ref{2}) to $\nabla' u^z$ yields
\begin{equation}\label{G}||\nabla\nabla'u^z||_{(0,\infty)}\lesssim ||\nabla' v^z||_{(0,\infty)} .\end{equation}
Now multiplying the equation (\ref{FraFor}) by $\partial_z^2$ 
\begin{equation}\label{E}\partial_z^2 u^z=-(-\Delta')^{\frac{1}{2}}\partial_z u^z+\partial_z v^z=-\Delta'u^z-(-\Delta')^{\frac{1}{2}}v^z+\partial_z v^z\end{equation}
and using the bandedness assumption in the form (\ref{Q}) we have 
\begin{equation}\label{i}
\begin{array}{rclc}
||\partial_z^2 u^z||_{(0,\infty)}&\leq& ||\nabla'^2 u^z||_{(0,\infty)}+||\nabla v^z||_{(0,\infty)}\\
&\stackrel{(\ref{G})}{\leq}&||\nabla v^z||_{(0,\infty)}\,.
\end{array}
\end{equation}

 The second term  of (\ref{C}), i.e $(-\Delta')^{-\frac{1}{2}}\partial_z(\partial_t-\partial_z^2)u^z$, can be bounded in the following way:
We multiply the equation  (\ref{FraFor}) by $(-\Delta')^{-\frac{1}{2}}(\partial_t-\partial_z^2)$

\begin{equation*}
    \left\{\begin{array}{rclc}
(\partial_z+(-\Delta')^{\frac{1}{2}})(-\Delta')^{-\frac{1}{2}}(\partial_t-\partial_z^2)u^z&=& (-\Delta')^{-\frac{1}{2}}(\partial_t-\partial_z^2)v^z \qquad & {\rm for } \quad   z>0,\\
(-\Delta')^{-\frac{1}{2}}(\partial_t-\partial_z^2)u^z&=&(-\Delta')^{-\frac{1}{2}}\partial_zv^z \qquad & {\rm for } \quad   z=0,
    \end{array}\right.
  \end{equation*}

where we have used  that at $z=0$

$$(\partial_t-\partial_z^2)u^z=-\partial_z^2 u^z\stackrel{(\ref{E})}{=}\partial_zv^z.$$
Applying (\ref{3}) to $(-\Delta')^{-\frac{1}{2}}(\partial_t-\partial_z^2)u^z$ and using the bandedness assumption in the form of (\ref{P}),

\begin{equation}\label{F}
||\nabla(-\Delta')^{-\frac{1}{2}}(\partial_t-\partial_z^2)u^z||_{(0,\infty)}\lesssim ||(-\Delta')^{-\frac{1}{2}}(\partial_t-\partial_z^2)v^z||_{(0,\infty)}+||\partial_zv^z||_{(0,\infty)}\,.
\end{equation}
%
%
%

Finally we can bound the last term of (\ref{C}), i.e $\partial_t u^z$:
We observe that $\partial_t u^z=(\partial_t -\partial_z^2 )u^z+\partial_z^2 u^z$ thus 

\begin{equation}\label{H}
||\partial_t u^z||_{(0,\infty)}\leq ||(\partial_t -\partial_z^2 )u^z||_{(0,\infty)}+||\partial_z^2 u^z||_{(0,\infty)}\,.
\end{equation}

For the first term in the right hand side of (\ref{H})
we notice that 
\begin{eqnarray*}
||(\partial_t-\partial_z^2)u^z||_{(0,\infty)}
&\stackrel{(\ref{P})}{\leq}&||(-\Delta')^{-\frac{1}{2}}\nabla'(\partial_t-\partial_z^2)u^z||_{(0,\infty)}\\
&\stackrel{(\ref{F})}{\lesssim}&||(-\Delta')^{-\frac{1}{2}}(\partial_t-\partial_z^2)v^z||_{(0,\infty)}+||\partial_z v^z||_{(0,\infty)}\\
&\lesssim&||(-\Delta')^{-\frac{1}{2}}(\partial_t-\partial_z^2)v^z||_{(0,\infty)}+||\nabla v^z||_{(0,\infty)}\,.
\end{eqnarray*}
The  second term  on the right hand side of (\ref{H}) is bounded in (\ref{i}). 
Thus we have the following bound for $\partial_t u$
\begin{equation}\label{M}||\partial_t u^z||_{(0,\infty)}\leq ||(-\Delta')^{-\frac{1}{2}}(\partial_t-\partial_z^2)v^z||_{(0,\infty)}+||\nabla v^z||_{(0,\infty)}\,.\end{equation}

Putting together all the above we obtain the desired estimate.
%
 
\item From the defining equation (\ref{Heat2}), the  basic estimate (\ref{4})
 and the bandedness assumption in form of (\ref{R}), we get

$$||(\partial_t-\partial_z^2)v'||_{(0,\infty)}+||\nabla'\nabla v'||_{(0,\infty)}\lesssim ||f'||_{(0,\infty)}\,.$$


\end{enumerate}
\end{proof}

\subsection{Proof of Theorem \ref{th1}}\label{ProofTh1}

Let $u,p,f$ be the solutions of the non-stationary Stokes equations in the strip $0<z<1$ (\ref{STOKES-STRIP}).
Then $\tilde u=\eta u,\tilde p=\eta p$ (with $\eta$ defined in (\ref{cutoff}) satisfy (\ref{UHS}), namely
 \begin{equation*}
  \left\{\begin{array}{rclc}
      \partial_t \tilde u-\Delta \tilde u+\nabla \tilde p &=& \tilde f \qquad & {\rm for } \quad z>0\,,\\
        \nabla\cdot \tilde u &=& \tilde\rho \qquad & {\rm for } \quad z>0 \,,\\
        \tilde u &=&  0 \qquad & {\rm for } \quad z=0\,,\\
         \tilde u &=& 0  \qquad  & {\rm for } \quad t=0 \,,\\
         \end{array}\right.        
 \end{equation*}
where 
\begin{equation}\label{Defi-bis}
\tilde f:=\eta f-2(\partial_z \eta)\partial_z u-(\partial_z^2\eta )u+(\partial_z\eta )pe_z, \qquad \qquad  \tilde\rho:=(\partial_z\eta )u^z\,.
\end{equation}
Since, by assumption $f,\rho$  are horizontally band-limited , 
then also $\tilde f$ and $\tilde \rho$  satisfy the horizontal bandedness assumption (\ref{BC1}) and (\ref{BC2}) respectively.
We can therefore apply Proposition \ref{pr1} to the upper half space problem (\ref{UHS}) and get
\begin{eqnarray*}
&&||(\partial_t -\partial_z^2)\tilde{u}'||_{(0,\infty)}+ ||\nabla'\nabla \tilde{u}'||_{(0,\infty)}+||\partial_t \tilde{u}^z||_{(0,\infty)}+||\nabla^2 \tilde{u}^z||_{(0,\infty)}+||\nabla \tilde{p}||_{(0,\infty)}\\
&\lesssim&||\tilde{f}||_{(0,\infty)}+||(-\Delta')^{-\frac{1}{2}}\partial_t \tilde{\rho}||_{(0,\infty)}+||(-\Delta')^{-\frac{1}{2}}\partial_z^2 \tilde{\rho} ||_{(0,\infty)}+||\nabla \tilde{\rho}||_{(0,\infty)}\,.
\end{eqnarray*}
By symmetry, we also have the same maximal regularity estimates in the lower half space. Indeed, let
 $\tilde{\tilde u} ,\tilde{\tilde p}$ satisfy the equation 
 \begin{equation}\label{LHS}
  \left\{\begin{array}{rclc}
      \partial_t \tilde{\tilde u}-\Delta \tilde{\tilde u}+\nabla \tilde {\tilde p} &=& \tilde{\tilde{f}} \qquad & {\rm for } \quad z<1\,,\\
        \nabla\cdot \tilde {\tilde u} &=& \tilde{\tilde\rho} \qquad & {\rm for } \quad z<1 \,,\\
        \tilde {\tilde u} &=&  0 \qquad & {\rm for } \quad z=1\,,\\
         \tilde {\tilde u} &=& 0  \qquad  & {\rm for } \quad t=0 \,,\\
         \end{array}\right.        
 \end{equation}
where 
\begin{equation}\label{Defi2}
\tilde{\tilde f}:=(1-\eta) f-2(\partial_z (1-\eta))\partial_z u-(\partial_z^2(1-\eta) )u+(\partial_z(1-\eta) )pe_z, \qquad   \tilde{\tilde\rho}:=(\partial_z(1-\eta) )u^z\,.
\end{equation}
Again by Proposition \ref{prop3} we have 
\begin{eqnarray*}
&&||(\partial_t -\partial_z^2)\tilde{\tilde{u}}'||_{(-\infty,1)}+ ||\nabla'\nabla \tilde{\tilde{u}}'||_{(-\infty,1)}+ ||\partial_t \tilde{\tilde{u}}^z||_{(-\infty,1)}+||\nabla^2 \tilde{\tilde{u}}^z||_{(-\infty,1)}+||\nabla \tilde{\tilde p}||_{(-\infty,1)}\\
&\lesssim&||\tilde{\tilde{f}}||_{(-\infty,1)}+||(-\Delta')^{-\frac{1}{2}}\partial_t \tilde{\tilde{\rho}}||_{(-\infty,1)}+||(-\Delta')^{-\frac{1}{2}}\partial_z^2 \tilde{\tilde{\rho}} ||_{(-\infty,1)}+||\nabla \tilde{\tilde{\rho}}||_{(-\infty,1)},
\end{eqnarray*}
where $||\cdot||_{(-\infty,1)}$ is the analogue of (\ref{NORM-HALF}) (see Section (\ref{notations}) for notations).
Since $u=\tilde u+\tilde{\tilde u}$ in the strip $[0,L)^{d-1}\times (0,1)$, by the triangle inequality and using the maximal regularity estimates above, we get
\begin{eqnarray*}
&&||(\partial_t -\partial_z^2)u'||_{(0,1)}+||\nabla'\nabla u'||_{(0,1)}+||\partial_t u^z||_{(0,1)}+||\nabla^2 u^z||_{(0,1)}+||\nabla p||_{(0,1)}\\
&\lesssim&||(\partial_t -\partial_z^2)\tilde{u}'||_{(0,\infty)}+||(\partial_t -\partial_z^2)\tilde{\tilde{u}}'||_{(-\infty,1)}+||\nabla'\nabla \tilde{u}'||_{(0,\infty)}+||\nabla'\nabla \tilde{\tilde{u}}'||_{(-\infty,1)}\\
&+&||\partial_t \tilde{u}^z||_{(0,\infty)}+ ||\partial_t \tilde{\tilde{u}}^z||_{(-\infty,1)}+ ||\nabla^2 \tilde{u}^z||_{(0,\infty)}+ ||\nabla^2 \tilde{\tilde{u}}^z||_{(-\infty,1)}\\
&+&||\nabla \tilde{p}||_{(0,\infty)}+||\nabla \tilde{\tilde p}||_{(-\infty,1)}\\
&\lesssim&||\tilde{f}||_{(0,\infty)}+||\tilde{\tilde{f}}||_{(-\infty,1)}+||(-\Delta')^{-\frac{1}{2}}\partial_t \tilde{\rho}||_{(0,\infty)}+||(-\Delta')^{-\frac{1}{2}}\partial_t \tilde{\tilde{\rho}}||_{(-\infty,1)}\\
&+&||(-\Delta')^{-\frac{1}{2}}\partial_z^2 \tilde{\rho} ||_{(0,\infty)}+||(-\Delta')^{-\frac{1}{2}}\partial_z^2 \tilde{\tilde{\rho}} ||_{(-\infty,1)}+||\nabla \tilde{\rho}||_{(0,\infty)}+||\nabla \tilde{\tilde{\rho}}||_{(-\infty,1)}\,.
\end{eqnarray*}
By the definitions of $\tilde{f}$ and $\tilde{\tilde{f}}$ we get

$$||\tilde{f}||_{(0,\infty)}+||\tilde{\tilde{f}}||_{(-\infty,1)}\lesssim ||f||_{(0,1)}+||\partial_z u||_{(0,1)}+||u||_{(0,1)}+||p||_{(0,1)}$$
and similarly for $\tilde{\rho}$ and $\tilde{\tilde{\rho}}$ we have
$$||\nabla \tilde{\rho}||_{(0,\infty)}+||\nabla \tilde{\tilde{\rho}}||_{(-\infty,1)}\lesssim ||\nabla u||_{(0,1)}+||u||_{(0,1)}$$
$$ ||(-\Delta')^{-\frac{1}{2}}\partial_t \tilde{\rho}||_{(0,\infty)}+||(-\Delta')^{-\frac{1}{2}}\partial_t \tilde{\tilde{\rho}}||_{(-\infty,1)}\lesssim || (-\Delta')^{-\frac{1}{2}}\partial_t u||_{(0,1)} $$
and 
\begin{eqnarray*}
&&||(-\Delta')^{-\frac{1}{2}}\partial_z^2 \tilde{\rho} ||_{(0,\infty)}+||(-\Delta')^{-\frac{1}{2}}\partial_z^2 \tilde{\tilde{\rho}} ||_{(-\infty,1)} \\
&\lesssim&||(-\Delta')^{-\frac{1}{2}}u^z||_{(0,1)}+||(-\Delta')^{-\frac{1}{2}}\partial_zu^z||_{(0,1)}+||(-\Delta')^{-\frac{1}{2}}\partial^2_zu^z||_{(0,1)}\,.
\end{eqnarray*}
Therefore, collecting the estimates, we have 
\begin{eqnarray*}
&&||(\partial_t -\partial_z^2)u'||_{(0,1)}+||\nabla'\nabla u'||_{(0,1)}+||\partial_t u^z||_{(0,1)}+||\nabla^2 u^z||_{(0,1)}+||\nabla p||_{(0,1)}\\
&\lesssim&||f||_{(0,1)}+||p||_{(0,1)}+||\nabla u||_{(0,1)}+||u||_{(0,1)}\\
&+& || (-\Delta')^{-\frac{1}{2}}\partial_t u||_{(0,1)}+||(-\Delta')^{-\frac{1}{2}}u^z||_{(0,1)}+||(-\Delta')^{-\frac{1}{2}}\partial_zu^z||_{(0,1)}+||(-\Delta')^{-\frac{1}{2}}\partial^2_zu^z||_{(0,1)}\,.
\end{eqnarray*}
Incorporating the horizontal bandedness assumption we find
 \begin{eqnarray*}
 ||\partial_z u||_{(0,1)}&\leq& R  ||\nabla'\partial_z u||_{(0,1)}\,,\\
 ||u||_{(0,1)}&\leq& R^2||(\nabla')^2 u||_{(0,1)} \,, \\    
 ||p||_{(0,1)}&\leq& R||\nabla'p||_{(0,1)} ,\\
 ||\nabla u||_{(0,1)}&\leq& R||\nabla'\nabla u||_{(0,1)},\\
 || (-\Delta')^{-\frac{1}{2}}\partial_t u||_{(0,1)}&\leq& R || \partial_t u||_{(0,1)}\,,\\
 ||(-\Delta')^{-\frac{1}{2}}u^z||_{(0,1)}&\leq&  R^3||\nabla'^2u^z||_{(0,1)}\,,\\
 ||(-\Delta')^{-\frac{1}{2}}\partial_zu^z||_{(0,1)}&\leq& R^2 ||\nabla'\partial_zu^z||_{(0,1)}\,,\\
 ||(-\Delta')^{-\frac{1}{2}}\partial^2_zu^z||_{(0,1)}&\leq& R||\partial^2_zu^z||_{(0,1)} \,.      
 \end{eqnarray*}               
Thus, for $R<R_0$ where $R_0$ is sufficiently small, all the terms in the right hand side,
except $f$ can be absorbed into the left hand side and the conclusion follows.

\section{Proof of main technical lemmas}\label{tec}

\begin{remark}\label{TIME}
 In the proof of Lemma \ref{lemma1}, Lemma \ref{lemma2} and Lemma \ref{lemma3} we  will derive inequalities between quantities  where $t$ is integrated between $0$ and $\infty$. 
  From the proof it is clear that the same inequalities are true  with $t$ integrated between $0$ and $t_0$ with constants that
  are not depending on $t_0$. Therefore
  dividing by $t_0$ and taking $\limsup_{t_{0}\rightarrow \infty}$ (see (\ref{LTaHA})) we shall obtain the desired estimates in terms of the interpolation norm (\ref{NORM-HALF}).
\end{remark}

\subsection{Proof of Lemma \ref{lemma1}}

\begin{proof}[Proof of Lemma \ref{lemma1}]\ \\
	      In order to simplify the notations, in what follows we will omit
	      the dependency of the functions from the time variable.	      
              It is enough to show 
	      \begin{equation*}
	       \label{U}||\nabla' u||_{(0,\infty)}\lesssim ||f||_{(0,\infty)},
	      \end{equation*}
              since, by equation (\ref{I}) $\partial_z u=(-\Delta')^{\frac{1}{2}} u+f$.
              We claim that, in order to prove (\ref{U}), it is enough to show
              \begin{equation}\label{V} 
	       \sup_z\langle|\nabla' u|\rangle'\lesssim \sup_z \langle|f|\rangle'
	      \end{equation}
              and 
              \begin{equation}\label{Z} 
	       ||\nabla'u||_{(0,\infty)}\lesssim \int\langle|f|\rangle'\frac{dz}{z}\,.
	      \end{equation}
               Indeed, by definition of the norm $||\cdot||_{(0,\infty)}$ (see (\ref{NORM-HALF}))
               if we select an arbitrary decomposition $\nabla'u=\nabla'u_1+\nabla'u_2$, where $u_1$ and $u_2$ are 
               solutions of the problem (\ref{I}) with right hand sides $f_1$ and $f_2$ respectively, we have 
              
              \begin{eqnarray*}
               ||\nabla' u||_{(0,\infty)}&\leq& ||\nabla' u_1||_{(0,\infty)}+\sup_{z}\langle|\nabla' u_2|\rangle'\\
               &\leq& \int\langle| f_1|\rangle'\frac{dz}{z}+\sup_z \langle|f_2|\rangle'\,.\\
              \end{eqnarray*}                           
               Passing to the infimum over all the decompositions of $f$ we obtain
               $$||\nabla' u||_{(0,\infty)}\lesssim ||f||_{(0,\infty)}.$$
              We recall that by Duhamel's principle we have the following representation
              \begin{equation}\label{DU}
              u(x',z)=\int_z^{\infty} u_{x',z_0}(z)dz_0,
              \end{equation}
              where $u_{z_0}$  is the harmonic extension of $f(\cdot,z_0)$ onto $\{z<z_0\}$, i.e it solves the boundary value problem 

              \begin{equation}\label{FraBackDu}
                \left\{\begin{array}{rclc}
                      (\partial_z-(-\Delta')^{\frac{1}{2}}) u_{z_0}&=&0  \qquad & {\rm for } \quad z<z_0\,,\\
                       u_{z_0}&=&f  \qquad & {\rm for } \quad z=z_0\,.\\
                       \end{array}\right.        
               \end{equation}                
                     \underline{Argument for (\ref{V}):} 
                     \newline
                     Using the representation of the solution of (\ref{FraBackDu}) via the Poisson kernel, i.e
                     $$u_{z_0}(x',z)=\int\frac{z_0-z}{(|x'-y'|^2+(z_0-z)^2)^{\frac{d}{2}}}f(x',z_0) dy'$$
                     we obtain the following bounds
                     \begin{equation}\label{AA}
                     \langle|\nabla' u_{z_0}(\cdot,z)|\rangle'\lesssim 
                     \left\{\begin{array}{lll}&&\langle|\nabla' f(\cdot,z_0)|\rangle',\\                    
                      &\frac{1}{(z_0-z)}&\langle| f(\cdot,z_0)|\rangle',\\                                       
                      &\frac{1}{(z_0-z)^2}&\langle|\nabla'(-\Delta')^{-1}f(\cdot,z_0)|\rangle'.\\                                       
                      \end{array}\right.                                       
                      \end{equation}                                      
		     By using the bandedness assumption in the form of (\ref{BAND1}) and (\ref{BAND2}), we have   
		     $$\langle|\nabla' u_{z_0}(\cdot,z)|\rangle'\lesssim \min \left\{\frac{1}{R}, \frac{R}{(z_0-z)^2}\right\}\langle|f(\cdot,z_0)|\rangle',$$
		     hence		     
                     \begin{eqnarray*}
			\langle|\nabla' u(\cdot,z)|\rangle'&\lesssim& \int_z^{\infty}\min \left\{\frac{1}{R},\frac{R}{(z_0-z)^2}\right\}\langle|f(\cdot,z_0)|\rangle'dz_0\\
			&\lesssim&\sup_{z_0\in (0,\infty)}\langle|f(\cdot,z_0)|\rangle'\int_z^{\infty}\min \left\{\frac{1}{R},\frac{R}{(z_0-z)^2}\right\}dz_0\\
			&\lesssim&\sup_{z_0\in (0,\infty)}\langle|f(\cdot,z_0)|\rangle',
	             \end{eqnarray*}	            
                     which, passing to the supremum in $z$, implies (\ref{V}).
                     \newline
                    From the above and applying Fubini's rule, we also have
                     \begin{align*}                       
			\int_{0}^{\infty}\langle|\nabla' u(\cdot,z)|\rangle' dz&\leq \int_0^{\infty} \int_z^{\infty}\min \left\{\frac{1}{R},\frac{R}{(z_0-z)^2}\right\}\langle|f(\cdot,z_0)|\rangle'dz_0 dz\numberthis\label{UNW1}\\
			&\leq\int_0^{\infty} \int_0^{z_0}\min \left\{\frac{1}{R},\frac{R}{(z_0-z)^2}\right\}dz\langle|f(\cdot,z_0)|\rangle'dz_0 \nonumber\\
			&\lesssim\int_0^{\infty}\langle|f(\cdot,z)|\rangle' dz \nonumber \,.\\
                     \end{align*}                     
		     \underline{Argument for (\ref{Z}):}
		     \newline
                      Let us consider $\chi_{2H\leq z\leq 4H}f$ where $\chi_{2H\leq z\leq 4H}$ is the characteristic function
                      on the interval $[2H,4H]$ and let $u_H$ be the solution to
                      $$(\partial_z-(-\Delta')^{\frac{1}{2}})u_{H}=\chi_{2H\leq z\leq 4H}f.$$
		    We claim
                     \begin{equation}\label{X}
			\sup_{z\leq H}\langle|\nabla' u_{H}|\rangle'\leq \int_0^{\infty}\langle|\chi_{2H\leq z\leq 4H}f|\rangle'\frac{dz}{z}
                     \end{equation}                    
		     and		     
		     \begin{equation}\label{XX}
			\int_H^{\infty}\langle|\nabla' u_{H}|\rangle'\frac{dz}{z}\leq \int_0^{\infty}\langle|\chi_{2H\leq z\leq 4H}f|\rangle'\frac{dz}{z}\,.
                     \end{equation}                          
                      From estimate (\ref{X}) and (\ref{XX}) the statement (\ref{Z}) easily follow.
                      Indeed, choosing $H=2^{n-1}$ and summing up over the dyadic intervals, we have
                      \begin{eqnarray*}
                      ||\nabla'u||&\leq& \sum_{n\in \Z}||\nabla'u_{2^{n-1}}||_{(0,\infty)}\\
                      &\leq&\sup_{z\leq 2^{n-1}}\langle|\nabla' u_{2^{n-1}}|\rangle'+\int_{2^{n-1}}^{\infty}\langle|\nabla' u_{2^{n-1}}|\rangle'\frac{dz}{z}\\
                      &\leq&\sum_{n\in \Z} \int_0^{\infty}\langle| \chi_{2^n\leq z\leq 2^{n+1}}f|\rangle'\frac{dz}{z}\\
                      &=&\int_0^{\infty}\langle|f|\rangle'\frac{dz}{z}\,.
                      \end{eqnarray*}                         
                        Argument for (\ref{X}):  
                                Fix $z\leq H$. Then, we have
                                  \begin{eqnarray*}
				      \langle|\nabla'u_H|\rangle'&\stackrel{(\ref{AA})}{\leq}&\int_z^{\infty}\frac{1}{(z_0-z)}\langle|\chi_{2H\leq z\leq 4H}f(\cdot,z_0)|\rangle' dz_0\\
 				      &\lesssim&\int_{2H}^{4H}\frac{1}{(z_0-z)}\langle|\chi_{2H\leq z\leq 4H}f(\cdot,z_0)|\rangle' dz_0\\ 
 				      &\lesssim& \frac{1}{H}\int_{2H}^{4H}\langle|\chi_{2H\leq z\leq 4H}f(\cdot,z_0)|\rangle' dz_0\\                                 
 				      &\leq&\int_{2H}^{\infty}\langle|\chi_{2H\leq z\leq 4H}f(\cdot,z_0)|\rangle' \frac{dz_0}{z_0}\\
 				      &\leq&\int_{0}^{\infty}\langle|\chi_{2H\leq z\leq 4H}f(\cdot,z_0)|\rangle' \frac{dz_0}{z_0}\,.\\
                                  \end{eqnarray*}                                
                                  Taking the supremum over all $z$ proves (\ref{X}). 
                                  \newline
                                 Argument for (\ref{XX}):
                                  For $z\geq H$ we have
                                 \begin{eqnarray*}
                                  \int_H^{\infty}\langle|\nabla'u_H|\rangle'\frac{dz}{z}&\lesssim& \frac{1}{H}\int_0^{\infty}\langle|\nabla'u_H|\rangle' dz\\
                                  &\stackrel{(\ref{UNW1})}{\lesssim}&\frac{1}{H}\int_0^{\infty}\langle|\chi_{2H\leq z\leq 4H}f|\rangle' dz\\
                                  &=&\frac{1}{H}\int_{2H}^{4H}\langle|\chi_{2H\leq z\leq 4H}f|\rangle' dz\\
                                  &\lesssim& \int_0^{\infty}\langle|\chi_{2H\leq z\leq 4H}f|\rangle'\frac{dz}{z}\,.\\
                                 \end{eqnarray*}
\end{proof}

\subsection{Proof of Lemma \ref{lemma2}}

\begin{proof}[Proof of Lemma \ref{lemma2}]\ \\		      
                     Let us first assume $g=0$. It is enough to show            	             
	             \begin{equation}\label{SupEst}
	             \sup_z\langle|\nabla'u|\rangle'\lesssim \sup_{z}\langle|f|\rangle'
	             \end{equation}	             
	             and	             
	             \begin{equation}\label{WeiEst}
	             \int_0^{\infty}\langle|\nabla'u|\rangle'\frac{dz}{z}\lesssim \int_{0}^{\infty}\langle|f|\rangle'\frac{dz}{z}\,.
	             \end{equation}	             
	             Recall that by Duhamel's principle we have the following representation
                     \begin{equation}\label{Du}
                      u(z)=\int_0^z u_{z_0}(\cdot,z)dz_0,
                     \end{equation}
                     where $u_{z_0}$  is the harmonic extension of $f(z_0)$ onto $\{z>z_0\}$, i.e it solves the boundary value problem         
                      \begin{equation}\label{FraForDu}
                        \left\{\begin{array}{rclc}
			  (\partial_z+(-\Delta')^{\frac{1}{2}}) u_{z_0}&=&0 \qquad & {\rm for } \quad z>z_0\,,\\
			   u_{z_0}&=&f \qquad & {\rm for } \quad z=z_0\,.\\
		           \end{array}\right.        
                      \end{equation}                   
                    From the Poisson's kernel representation we learn that 
                    $$\langle|\nabla' u_{z_0}(\cdot,z)|\rangle'\lesssim \begin{cases}
                                                                  \langle|\nabla' f(\cdot,z_0)|\rangle'\,,\\
                                                                  \frac{1}{(z-z_0)^2}\langle|\nabla'(-\Delta')^{-1}f(\cdot,z_0)|\rangle'\,.
                                                                  \end{cases}$$
                    Using the bandedness assumption in the form of (\ref{BAND1}) and (\ref{BAND2}) 
                    $$\langle|\nabla' u_{z_0}(\cdot,z)|\rangle'\lesssim \min\left\{\frac{1}{R},\frac{R}{(z-z_0)^2}\right\}\langle| f(\cdot,z_0)|\rangle'$$
                    and observing  (\ref{Du}), we obtain 
                    \begin{equation}\label{AUX} 
                    \begin{array}{rclc}     
                    \langle|\nabla' u(\cdot,z)|\rangle'                               
		      &\lesssim& \int_0^{z}\min\left\{\frac{1}{R},\frac{R}{(z-z_0)^2}\right\}\langle| f(\cdot,z_0)|\rangle' dz_0\\
		       &\leq&\sup_{z_0}\langle| f(\cdot,z_0)|\rangle'\int_0^{z}\min\left\{\frac{1}{R},\frac{R}{(z-z_0)^2}\right\} dz_0\\
		       &\lesssim&\sup_{z_0}\langle| f(\cdot,z_0)|\rangle'\,. 
                    \end{array}    
                    \end{equation}                                                       
                    Estimate (\ref{SupEst}) follows from (\ref{AUX}) by passing to the supremum in $z$.     
                    \newline
                    From the above (\ref{AUX}), multiplying by the weight $\frac{1}{z}$ and observing that $z>z_0$ we have                    
                   \begin{equation}\label{AUX2}
		      \langle|\nabla' u(\cdot,z)|\rangle'\frac{1}{z}\lesssim \int_0^{z}\min\left\{\frac{1}{R},\frac{R}{(z-z_0)^2}\right\}\langle| f(\cdot,z_0)|\rangle'\frac{dz_0}{z_0}\,.
                    \end{equation}                   
                    After integrating in $z\in (0,\infty)$ and applying Young's estimate  we get (\ref{WeiEst}).                                      
		     \vspace{1cm}
		     
                     Let's assume now the general case, with $g\neq 0$. We want to prove (\ref{3}). 
                     Recall that by definition $\tilde g(x',z):=g(x')$ and consider $u-\tilde g$. 
                     By construction it satisfies                                                 
                     \begin{equation*}
                        \left\{\begin{array}{rclc}
			  (\partial_z+(-\Delta')^{-\frac{1}{2}})(u-\tilde g)&=&f-(-\Delta')^{-\frac{1}{2}}g  \qquad & {\rm for } \quad z>0\,,\\
			  u-\tilde g&=&0   \qquad & {\rm for } \quad z=0 \,.\\ 
			     \end{array}\right.        
                      \end{equation*}                 
                    Using the first part of the proof of (\ref{2}) and triangle inequality, we have                  
                    $$||\nabla u||_{(0,\infty)}\lesssim ||\nabla \tilde g||_{(0,\infty)}+||f||_{(0,\infty)}+||(-\Delta')^{\frac{1}{2}}\tilde g||_{(0,\infty)}\,.$$                    
                    Therefore by the bandedness assumption in the form of (\ref{Q}) we can conclude (\ref{3}).                   
  \end{proof}   
                 
 \subsection{Proof of Lemma \ref{lemma3}}

\begin{proof}[Proof of Lemma \ref{lemma3}]\ \\         
		  We will show that, for the non-homogeneous heat 
		  equation with Dirichlet boundary condition   
		    \begin{equation}\label{A1}
                        \left\{\begin{array}{rclc}
			(\partial_t-\Delta)u&=&f  \qquad & {\rm for } \quad   z>0\,,\\
			u&=&0  \qquad & {\rm for } \quad    z=0\,,\\
			u&=&0  \qquad & {\rm for } \quad    t=0\,,\\
		     	 \end{array}\right.        
                      \end{equation}   
                      we have the following estimates
                    \begin{equation}\label{A1.1}
		     \left\langle\int|(\partial_t-\partial_z^2)u(\cdot,z,\cdot)|\,\frac{dz}{z}\right\rangle+\left\langle\int|\nabla'^2 u(\cdot,z,\cdot)|\frac{dz}{z}\right\rangle\lesssim \left\langle\int|f(\cdot,z,\cdot)| \frac{dz}{z}\right\rangle\,,
		    \end{equation}             
		   \begin{equation}\label{A1.2}
		    \left\langle|\nabla'\partial_z u(\cdot,z,\cdot)|_{z=0}\right\rangle\lesssim \left\langle\int|f(\cdot,z,\cdot)| \frac{dz}{z}\right\rangle\,,
		   \end{equation} 
		   \begin{equation}\label{A1.3}
		    \left\langle\sup_{z}|\nabla'^2 u(\cdot,z,\cdot)|\right\rangle\lesssim \left\langle\sup_{z}|f(\cdot,z,\cdot)|\right\rangle \,,
		   \end{equation}               
		   \begin{equation}\label{A1.4}
		    \left\langle\sup_{z}|\nabla'\partial_z u(\cdot,z,\cdot)|\right\rangle\lesssim \left\langle\sup_{z}|f(\cdot,z,\cdot)|\right\rangle \,.
		   \end{equation} 
		   In order to bound the off-diagonal components 
		   of the Hessian, we consider the 
		   decomposition
		   \begin{equation}\label{SPLIT}
		    u=u_N+u_C,
		   \end{equation} 
		 where $u_N$ solves   
		  \begin{equation}\label{A2}
                        \left\{\begin{array}{rclc}
		        (\partial_t-\Delta)u_N &=&f \qquad & {\rm for } \quad   z>0\,,\\
		        \partial_z u_N &=& 0 \qquad & {\rm for } \quad    z=0\,,\\
		          u_N &=& 0 \qquad & {\rm for } \quad    t=0\,,\\
		        \end{array}\right.        
                   \end{equation}
                     and $u_C$ solves
		   \begin{equation}\label{A3}
                        \left\{\begin{array}{rclc}
		        (\partial_t-\Delta)u_C&=&0   \qquad & {\rm for } \quad   z>0\,,\\
		        \partial_z u_C &=&\partial_z u  \qquad & {\rm for } \quad    z=0\,,\\
		        u_C &=&0  \qquad & {\rm for } \quad    t=0\,.\\
		        \end{array}\right.        
                   \end{equation}                       		       
		The splitting (\ref{SPLIT}) is valid by the uniqueness of the Neumann problem.
                For the auxiliary problems  (\ref{A2}) and (\ref{A3}) we have  the following bounds      		               
		\begin{equation}\label{A2.1}
		 \left\langle \int|\nabla'\partial_z u_N(\cdot,z,\cdot)|\frac{dz}{z}\right\rangle \lesssim \left\langle\int|f(\cdot,z,\cdot)| \frac{dz}{z}\right\rangle\,,
		\end{equation}       
       		\begin{equation}\label{A3.1}
		  \left\langle\sup_z|\nabla'\partial_z u_C(\cdot,z,\cdot)|\right\rangle\lesssim \left\langle|\nabla'\partial_z u(\cdot,z,\cdot)|_{z=0}\right\rangle \,.
		\end{equation}       
		We claim that estimates (\ref{A1.1}), (\ref{A1.2}),(\ref{A1.3}), (\ref{A1.4}), (\ref{A2.1}) and (\ref{A3.1})  yield (\ref{4}). 
		 \newline
		Let us first consider the bound for $\nabla'^2$. 
		Consider $u=u_1+u_2$, where $u_1$ and $u_2$ satisfy (\ref{A1}) with right hand side $f_1$ and $f_2$ respectively. 
		 We have	 
		\begin{eqnarray*}
		||\nabla'^2 u||_{(0,\infty)}&\lesssim &
		\left\langle\sup_z|\nabla'^2u_{1}|\right\rangle+\left\langle\int|\nabla'^2 u_{2}|\frac{dz}{z}\right\rangle\\
		& \stackrel{ (\ref{A1.1}) \& (\ref{A1.3})}{\lesssim} & \left\langle\sup_z|f_1|\right\rangle+\left\langle\int|f_2|\frac{dz}{z}\right\rangle,
		\end{eqnarray*}		
		which implies, upon taking infimum over all decompositions $f=f_1+f_2$
		\begin{equation}\label{AAA}
		 ||\nabla'^2 u||_{(0,\infty)}\lesssim ||f||_{(0,\infty)}.
		\end{equation}      
		 We now consider a further decomposition of $u_2$ , i.e $u_2=u_{2C}+u_{2N}$ where $u_{2C}$ satisfies (\ref{A3}) and $u_{2N}$ satisfies (\ref{A2}).
		 Therefore $u=u_1+u_{2C}+u_{2N}$ and we can bound the off-diagonal components of the Hessian 
                \begin{eqnarray*}
		||\nabla'\partial_z u||_{(0,\infty)}& \lesssim &
		 \left\langle\sup_z|\nabla'\partial_z u_1|\right\rangle+\left\langle\sup_z|\nabla'\partial_zu_{2C}|\right\rangle+\left\langle\int|\nabla'\partial_z u_{2N}|\frac{dz}{z}\right\rangle\\
		& \stackrel{(\ref{A1.2}),(\ref{A3.1}),(\ref{A2.1}) \& (\ref{A1.4})}{\lesssim} & \left\langle\sup_z|f_1|\right\rangle+\left\langle\int|f_2|\frac{dz}{z}\right\rangle\,.
		\end{eqnarray*}		
		From the last inequality, passing to the infimum over all the possible decompositions of $f$ we get
		\begin{equation}\label{BBB}
		||\nabla'\partial_z u||_{(0,\infty)}\lesssim ||f||_{(0,\infty)}.
		\end{equation}         		
		On one hand estimate (\ref{AAA}) and (\ref{BBB}) imply 
		$$||\nabla\nabla' u||_{(0,\infty)}\lesssim||\nabla'^2 u||_{(0,\infty)}+||\nabla'\partial_z u||_{(0,\infty)}\,,$$   		
		on the other hand equation (\ref{III}) and  estimate (\ref{AAA}) yield         
		$$||(\partial_t-\partial_z^2) u||_{(0,\infty)}\lesssim ||f||_{(0,\infty)}\,.$$		               
		   \underline{Argument for (\ref{A1.1})}
		   \newline
		   Let $u$ be a solution of problem of (\ref{A1}).
                   Keeping in mind Remark (\ref{TIME}) it is enough to show              
                   $$\int_0^{\infty}\int_0^{\infty}\langle|\nabla'^2u|\rangle'\frac{dz}{z}dt\lesssim \int_{0}^{\infty}\int_0^{\infty}\langle|f|\rangle'\frac{dz}{z}dt \,.$$    
		   By the Duhamel's principle we have     
		   \begin{equation}\label{Duhamel1}
                    u(x',z,t)=\int_{s=0}^{t} u_{s}(x',z,t)ds ,
                   \end{equation}
                   where $u_{s}$ is the solution to the homogeneous, initial value problem		   		  
	           \begin{equation}\label{Du-eq}	        
                   \left\{\begin{array}{rclc}
		   (\partial_t-\Delta)u_s&=&0  \qquad & {\rm for } \quad  z>0, t>s\,,\\
		    u_s&=&0  \qquad & {\rm for } \quad  z=0, t>s\,,\\
		    u_s&=&f  \qquad & {\rm for } \quad  z>0, t=s\,.\\
		    \end{array}\right.        
                   \end{equation}                                       
		 Extending $u$ and $f$ to the whole space by odd reflection 
		 \footnote{with abuse of notation we will call again $u$ and $f$ these extensions.},
		 we are left to study the problem       
		 \begin{equation*}
                 \left\{\begin{array}{rclc}
		 (\partial_t-\Delta)u_s&=&0 \qquad & {\rm for } \quad z\in \R, t>s\,,\\
		 u_s&=&f  \qquad & {\rm for } \quad z\in \R,  t=s\,,\\
	         \end{array}\right.        
                 \end{equation*}                    
		 the solution of which can be represented via heat kernel as     
                  \begin{equation}\label{representation}
		 \begin{array}{rclc}
		 u_s(x',z,t)&=&\int_{\R}\Gamma(\cdot,z-\tilde z, t-s)\ast_{x'}f(\cdot,\tilde z,s)d\tilde z\\
		 &=&\int_{0}^{\infty}\left[\Gamma(\cdot,z-\tilde z,t-s)-\Gamma ( \cdot,z+\tilde z,t-s)\right]\ast_{x'}f(\cdot,\tilde z,s)d\tilde z\,.\\
		 \end{array} 
                  \end{equation} 
		 The application of  $\nabla'^2$ to the representation above yields 
		 \begin{eqnarray*}
                  &&\nabla'^2u_s(x',z,t)\\
 		 &=&\footnotesize{\begin{cases}
 		     \int_{0}^{\infty}\int_{\R^{d-1}}\nabla'\Gamma_{d-1}(x'-\tilde{x'},t-s)\left(\Gamma_1(z-\tilde z,t-s)-\Gamma_1( z+\tilde z,t-s)\right)\nabla'f(\tilde{x'},\tilde z,s)d\tilde{x'}d\tilde z\,,\\
 		     \int_{0}^{\infty}\int_{\R^{d-1}}\nabla'^3\Gamma_{d-1}(x'-\tilde{x'},t-s)\left(\Gamma_1(z-\tilde z,t-s)-\Gamma_1 (z+\tilde z,t-s)\right)(-\Delta')^{-1}\nabla'f(\tilde{x'},\tilde z,s)d\tilde{x'}d\tilde z\,.\\
 		    \end{cases}}
 		    \end{eqnarray*}    		
 		Averaging in the horizontal direction we obtain, on the one hand
%
		{\footnotesize{
		 \begin{eqnarray*}
                  &&\langle|\nabla'^2u_s(\cdot,z,t)|\rangle'\\
                  &\lesssim&\int_{0}^{\infty}\langle|\nabla'\Gamma_{d-1} (\cdot,t-s)|\rangle'|\Gamma_1(z-\tilde z,t-s)-\Gamma_1 ( z+\tilde z,t-s)|\langle|\nabla'f(\cdot,\tilde z,s)|\rangle' d\tilde z\\
                  &\stackrel{(\ref{z0})\&(\ref{BAND2})}{\lesssim}&\int_{0}^{\infty}\frac{1}{(t-s)^{\frac{1}{2}}}|\Gamma_1(z-\tilde z,t-s)-\Gamma_1 ( z+\tilde z,t-s)|\frac{1}{R}\langle|f(\cdot,\tilde z,s)|\rangle' d\tilde z\\
		 \end{eqnarray*}}}
		 and, on the other hand
		  {\footnotesize{
		 \begin{eqnarray*}
                  &&\langle|\nabla'^2u_s(\cdot,z,t)|\rangle'\\
                   &\lesssim& \int_{0}^{\infty}\langle|\nabla'^3\Gamma_{d-1} (\cdot,t-s)\rangle'|\Gamma_1(z-\tilde z,t-s)-\Gamma_1( z+\tilde z,t-s)|\langle|(-\Delta')^{-1}\nabla'f(\cdot,\tilde z,s)|\rangle'd\tilde z\,\\
                   &\stackrel{(\ref{z0})\&(\ref{BAND1})}{\lesssim}&\int_{0}^{\infty}|\Gamma_1(z-\tilde z,t-s)-\Gamma_1( z+\tilde z,t-s)|\frac{1}{(t-s)^{\frac{3}{2}}} R\langle|f(\cdot,\tilde z,s)|\rangle'd\tilde z\,.
		 \end{eqnarray*}}}		 
		 Multiplying by the weight $\frac{1}{z}$ and integrating in $z\in(0,\infty)$ we get
		 \begin{equation*}
		\int_0^{\infty}\langle|\nabla'^2u_s(\cdot,t)|\rangle'\frac{dz}{z}\lesssim \left(\sup_{\tilde z}\int_0^{\infty}K_{t-s}(z,\tilde z) dz\right)
		 \footnotesize{\begin{cases}
		\frac{1}{(t-s)^{\frac{1}{2}}}\frac{1}{R}\int_{0}^{\infty}\langle|f(\cdot,\tilde z,s)|\rangle' \frac{d\tilde z}{\tilde z}\,,\\
		\frac{R}{(t-s)^{\frac{3}{2}}}\int_{0}^{\infty} \langle |f(x',\tilde z,s)|\rangle'\frac{d\tilde z}{\tilde z}\,,
		\end{cases}}
		\end{equation*}    
		where we called $K_{t-s}(z,\tilde z)=\frac{\tilde z}{z}|\Gamma_1(z-\tilde z,t-s)-\Gamma_1( z+\tilde z,t-s)|$.
		\newline
		 From Lemma \ref{Lemma5} we infer 
		 $$\sup_{\tilde z}\int_0^{\infty}K_{t-s}(z,\tilde z) dz\stackrel{(\ref{EE1})}{\lesssim} \int_{\R}|\Gamma_1(z,t-s)|dz+\sup_{z\in \R}(z^2|\partial_z\Gamma_1(z,t-s)|)\stackrel{(\ref{x1})\&(\ref{y3})}{\lesssim}1$$
                and therefore we have
		\begin{equation*}
		\int_0^{\infty}\langle|\nabla'^2u_s(\cdot,z,t)|\rangle'\frac{dz}{z}\lesssim
		 \footnotesize{\begin{cases}
		\frac{1}{(t-s)^{\frac{1}{2}}}\frac{1}{R}\int_{0}^{\infty}\langle|f(\cdot,\tilde z,s)|\rangle' \frac{d\tilde z}{\tilde z}\,,\\
		\frac{1}{(t-s)^{\frac{3}{2}}} R \int_{0}^{\infty}\langle |f(\cdot,\tilde z,s)|\rangle'\frac{d\tilde z}{\tilde z}\,.
		\end{cases}}
		\end{equation*}    
		Finally, inserting the previous estimate into the Duhamel formula (\ref{Duhamel1}) and integrating in time we get
               \begin{eqnarray}
               &&\int_0^{\infty}\langle|\nabla'^2 u(\cdot,z,t)|\rangle'\frac{dz}{z} dt \notag \\
               &\stackrel{(\ref{Duhamel1})}{\lesssim}&\int_0^{\infty}\int_0^t\langle|\nabla'^2 u_s(\cdot,z,t)|\rangle' \frac{dz}{z} ds dt\notag\\
               &\lesssim&\int_0^{\infty}\int_s^{\infty} \min\left\{\frac{1}{R(t-s)^{\frac{1}{2}}},\frac{R}{(t-s)^{\frac{3}{2}}}\right\}\int_0^{\infty}\langle|f(\cdot,\tilde z,s)|\rangle'\frac{d\tilde z}{\tilde z} dt ds\notag\\
               &\lesssim&\int_0^{\infty}\int_s^{\infty}\min\left\{\frac{1}{R(t-s)^{\frac{1}{2}}},\frac{R}{(t-s)^{\frac{3}{2}}}\right\}dt\int_0^{\infty}\langle|f(\cdot,\tilde z,s)|\rangle'\frac{d\tilde z}{\tilde z} ds \label{Es1}\\
               &\lesssim&\int_0^{\infty}\int_0^{\infty}\min\left\{\frac{1}{R\tau^{\frac{1}{2}}},\frac{R}{\tau^{\frac{3}{2}}}\right\}d\tau\int_0^{\infty}\langle|f(\cdot,\tilde z,s)|\rangle'\frac{d\tilde z}{\tilde z} ds\label{Es2},\\
               &\lesssim&\int_0^{\infty}\int_0^{\infty}\langle|f(\cdot,\tilde z,s)|\rangle'\frac{d\tilde z}{\tilde z} ds,\notag
               \end{eqnarray}
               where in the second to last inequality we used 
                \begin{equation}\label{MIN}
                \int_0^{\infty}\min\left\{\frac{1}{R\tau^{\frac{1}{2}}},\frac{R}{\tau^{\frac{3}{2}}}\right\}d\tau\lesssim 1\,.
                \end{equation}                 
		 \underline{Argument for (\ref{A1.2}):}  
		 \newline
		     Let $u$ be a solution of problem of (\ref{A1}). Recall that we need to prove
		  
		   \begin{equation}\label{BO1}
		      \int_0^{\infty}\langle|\nabla'\partial_z u|_{z=0}(\cdot,z,t)|\rangle' dt\lesssim \int_0^{\infty}\int_0^{\infty}\langle|f(\cdot,z,t)|\rangle' dt\frac{dz}{z}\,.
		   \end{equation}    		   
                  The solution of the equation (\ref{Du-eq}) extended to the whole space by odd reflection can be represented by (\ref{representation}) (see argument for (\ref{A1.1})). Therefore
		  \begin{eqnarray*}		
		  &&\nabla'\partial_zu_s(x',z,t)|_{z=0}\\  
		  &=&\begin{cases}
		  -2\int_{\R^{d-1}}\int_0^{\infty} \Gamma_{d-1}(x'-\tilde{x'},t-s)\partial_z\Gamma_1(\tilde z,t-s)\nabla'f(\tilde {x'},\tilde{z},s)d\tilde{x'}d\tilde z\,,\\
		  -2\int_{\R^{d-1}}\int_0^{\infty}\nabla'\Gamma_{d-1}(x'-\tilde{x'},t-s) \partial_z\Gamma_1(\tilde z,t-s)\nabla'(-\Delta')^{-1}\nabla'f(\tilde{x'},\tilde{z},s) d\tilde{x'}d\tilde z\,.
		  \end{cases}
		  \end{eqnarray*}          
		  Taking the horizontal average we get, on the one hand
		  \begin{eqnarray*}
		    &&\langle|\nabla'\partial_zu_s(\cdot,z,t)|_{z=0}|\rangle'\\
		    &\lesssim& \int_0^{\infty}\langle|\Gamma_{d-1}(\cdot,t-s)|\rangle'|\partial_z \Gamma_1(\tilde z,t-s)|\langle|\nabla'f(\cdot,\tilde z,s)|\rangle' d\tilde z\\
		    &\stackrel{(\ref{z0})}{\lesssim}&\int_0^{\infty}|\partial_z \Gamma_1(\tilde z,t-s)|\langle|\nabla' f(\cdot,\tilde z,s)|\rangle'd\tilde z\\
		    &\stackrel{(\ref{BAND2})}{\lesssim}&\frac{1}{R}\int_0^{\infty}|\partial_z \Gamma_1(\tilde z,t-s)|\langle|f(\cdot,\tilde z,s)|\rangle'd\tilde z\\
		    &\lesssim&\frac{1}{R}\sup_{\tilde z}|\tilde z\partial_z\Gamma_1(\tilde z,t-s)|\int_0^{\infty}\langle|f(\cdot,\tilde{z},s)|\rangle'\frac{d\tilde z}{\tilde z}
		  \end{eqnarray*}          
		  and on the other hand        
		  \begin{eqnarray*}
		    &&\langle|\nabla'\partial_zu_s(\cdot,z,t)|_{z=0}|\rangle'\\
		    &\lesssim& \int_0^{\infty}\langle|(\nabla')^2\Gamma_{d-1}(\cdot,t-s)|\rangle'|\partial_z \Gamma_1(\tilde z,t-s)|\langle|(-\Delta')^{-1}\nabla'f(\cdot,\tilde z,s)|\rangle' d\tilde z\\
		    &\stackrel{(\ref{z0})}{\lesssim}&\frac{1}{(t-s)}\int_0^{\infty}|\partial_z \Gamma_1(\tilde z,t-s)|\langle|(-\Delta')^{-1}\nabla' f(\cdot,\tilde z,s)|\rangle'd\tilde z\\
		    &\stackrel{(\ref{BAND1})}{\lesssim}&\frac{R}{(t-s)}\int_0^{\infty}|\partial_z \Gamma_1(\tilde z,t-s)|\langle|f(\cdot,\tilde z,s)|\rangle'd\tilde z\\
		    &\lesssim&\frac{R}{(t-s)}\sup_{\tilde z}|\tilde z\partial_z\Gamma_1(\tilde z,t-s)|\int_0^{\infty}\langle|f(\cdot,\tilde{z},s)|\rangle'\frac{d\tilde z}{\tilde z}\,.
		  \end{eqnarray*}                 
		Using the estimate (\ref{y2}) we get
		\begin{equation*}
		 \langle|\nabla'\partial_zu_s(x',z,t)|_{z=0}|\rangle' \lesssim \begin{cases}
		                                                                \frac{1}{(t-s)^{1/2}R}\int_0^{\infty}\langle|f(\cdot,\tilde z,s)|\rangle'\frac{d\tilde z}{\tilde z}\,,\\
		                                                                \frac{R}{(t-s)^{3/2}}\int_0^{\infty}\langle|f(\cdot,\tilde z,s)|\rangle'\frac{d\tilde z}{\tilde z}\,.
		                                                               \end{cases}
		                                                               \end{equation*} 		                                                                     
               Finally,  inserting into Duhamel's formula and integrating in time we have 
               \begin{eqnarray*}
               &&\int_0^{\infty}\langle|\nabla'\partial_z u(\cdot,z,t)|_{z=0}\rangle' dt \\
               &\stackrel{(\ref{Duhamel1})}{\lesssim}&\int_0^{\infty}\int_0^t\langle|\nabla'\partial_z u_s(\cdot,z,t)|_{z=0}\rangle' ds dt\\
               &\lesssim& \int_0^{\infty}\int_s^{\infty}\min\{\frac{1}{R(t-s)^{\frac{1}{2}}},\frac{R}{(t-s)^{\frac{3}{2}}}\}\int_0^{\infty}\langle|f(\cdot,\tilde z,s)|\rangle'\frac{d\tilde z}{\tilde z} dt ds\\
               &\stackrel{(\ref{Es1})\& (\ref{Es2})}{\lesssim}&\int_0^{\infty}\int_0^{\infty}\langle|f(x',z,s)|\rangle'\frac{d\tilde z}{\tilde z} ds.
               \end{eqnarray*}          
                 \underline{Argument for (\ref{A1.3}):}
                 \newline
                  Let $u$ be the solution of problem (\ref{A1}). 
                  We recall that we want to prove               
                  \begin{equation}\label{SUP}
                  \sup_z\int_0^{\infty}\langle|\nabla'^2u(\cdot,z,t)|\rangle'dt\lesssim \sup_z\int_{0}^{\infty}\langle|f(\cdot,z,t)|\rangle'dt \,.
                  \end{equation}   		   		  
                  The solution of equation (\ref{Du-eq}) extended to the whole space  can be represented by (\ref{representation})
                  (see argument for (\ref{A1.1})). Therefore
                  applying $\nabla'^2$ to (\ref{representation}) and considering the horizontal average we have, on the one hand 
                \begin{eqnarray*}
 		 &&\langle|\nabla'^2u_s(\cdot,z,t)|\rangle'\\
 		&\lesssim& \int_{\R}\langle|\nabla'\Gamma_{d-1}(\cdot,t-s)|\rangle'|\Gamma_1(z-\tilde z,t-s)|\langle|\nabla'f(\cdot,\tilde{z},s)|\rangle' d\tilde z\\
                &\stackrel{(\ref{z0})\& (\ref{BAND2})}{\lesssim}& \int_{\R}\frac{1}{(t-s)^{\frac{1}{2}}}|\Gamma_1(z-\tilde z,t-s)|\frac{1}{R}\langle|f(\cdot,\tilde{z},s)|\rangle' d\tilde z
                \end{eqnarray*}                
                and on the other hand 
                 \begin{eqnarray*}
 		 &&\langle|\nabla'^2u_s(\cdot,z,t)|\rangle'\\
 		 &\lesssim&\int_{\R}\langle|\nabla'^3\Gamma_{d-1}(\cdot,t-s)|\rangle'|\Gamma_1(z-\tilde z,t-s)|\langle|(-\Delta')^{-1}\nabla'f(\cdot,\tilde{z},s)|\rangle' d\tilde z\,\\ 
 		 &\stackrel{(\ref{z0}) \& (\ref{BAND1})}{\lesssim} &\int_{\R}\frac{1}{(t-s)^{\frac{3}{2}}}|\Gamma_1(z-\tilde z,t-s)|R\langle|f(\cdot,\tilde{z},s)|\rangle' d\tilde z\,.
		 \end{eqnarray*}                 
                  Inserting the above estimates in the Duhamel's formula (\ref{Duhamel1}), we have                 
                 {\footnotesize{\begin{eqnarray*}
                 &&\int_0^{\infty}\int_0^t \langle|\nabla'^2u_s(z,\cdot)|\rangle'dsdt\\
                 &\lesssim& \int_0^{\infty}\int_s^{\infty}\min\left\{\frac{1}{R(t-s)^{\frac{1}{2}}},\frac{R}{(t-s)^{\frac{3}{2}}}\right\}\int_{\R}|\Gamma_1(z-\tilde{z},t-s)|\langle|f(\cdot,\tilde{z},s)|\rangle'd\tilde z dsdt\\
                 &\lesssim&\int_{\R}\left(\int_0^{\infty}\min\left\{\frac{1}{R \tau^{\frac{1}{2}}},\frac{R}{\tau^{\frac{3}{2}}}\right\}|\Gamma_1(z-\tilde{z},\tau)|d\tau\right)\int_0^{\infty}\langle|f(\cdot,\tilde{z},s)|\rangle'dsd\tilde z\\
		 &\lesssim& \sup_{\tilde{z}}\int_0^{\infty}\langle|f(\cdot,\tilde{z},s)|\rangle'ds\int_{\R}\int_0^{\infty}\min\left\{\frac{1}{R \tau^{\frac{1}{2}}},\frac{R}{\tau^{\frac{3}{2}}}\right\}|\Gamma_1(z-\tilde{z},\tau)|d\tau d\tilde z\\
		 &\stackrel{(\ref{x1})}{\lesssim}&\sup_{\tilde{z}}\int_0^{\infty}\langle|f(\cdot,\tilde{z},s)|\rangle'ds\int_0^{\infty}\min\left\{\frac{1}{R \tau^{\frac{1}{2}}},\frac{R}{\tau^{\frac{3}{2}}}\right\}d\tau\int_{\R}|\Gamma_1(z-\tilde{z},\tau)|d\tilde z\\
		 &\stackrel{(\ref{MIN})}{\lesssim}&\sup_{\tilde{z}}\int_0^{\infty}\langle|f(\cdot,\tilde{z},s)|\rangle'ds\,.
                 \end{eqnarray*}}}
                 Taking the supremum in $z$ we obtain the desired estimate.          
                 \newline       
                  \underline{Argument for (\ref{A1.4}):}
                  \newline
                  Let $u$ be the solution of problem (\ref{A1}). 
                   We claim               
                   \begin{equation}\label{SUP2}
                   \sup_z\int_0^{\infty}\langle|\nabla'\partial_zu|\rangle'dt\lesssim \sup_z\int_{0}^{\infty}\langle|f|\rangle'dt \,.
                   \end{equation}   		 
		  The solution of the equation (\ref{Du-eq}) extended to the whole space  can be represented by (see argument for (\ref{A1.1}))      
		 \begin{align*}
		  u_{s}(x',z,t)&=\int_{\R}\Gamma(\cdot,z-\tilde z, t-s)\ast_{x'}f(\cdot,\tilde z,s)d\tilde z\,.
		 \end{align*}          
		 Applying $\nabla'\partial_z$ and considering the horizontal average we obtain, on the one hand 		 
                 \begin{eqnarray*}
 		 &&\langle|\nabla'\partial_z u_s(\cdot,z,t)|\rangle'\\ 		 
 		 &\lesssim&\int_{\R}\langle|\Gamma_{d-1}(\cdot,t-s)|\rangle'|\partial_z\Gamma_1(z-\tilde z,t-s)|\langle|\nabla'f(\cdot,\tilde{z},s)|\rangle' d\tilde z\\
 		 &\stackrel{(\ref{BAND2})}{\lesssim}&\int_{\R}|\partial_z\Gamma_1(z-\tilde z,t-s)|\frac{1}{R}\langle|f(\cdot,\tilde{z},s)|\rangle' d\tilde z\\
 		 \end{eqnarray*}
 		 and, on the other hand 
 		 \begin{eqnarray*}
 		 &&\langle|\nabla'\partial_z u_s(\cdot,z,t)|\rangle'\\ 		 
 		 &\lesssim&\int_{\R}\langle|\nabla'^2\Gamma_{d-1}(\cdot,t-s)|\rangle'|\partial_z\Gamma_1(z-\tilde z,t-s)|\langle|(-\Delta')^{-1}\nabla'f(\cdot,\tilde{z},s)|\rangle' d\tilde z\\  
 		 &\stackrel{(\ref{BAND1})}{\lesssim}&\int_{\R}\frac{1}{(t-s)}|\partial_z\Gamma_1(z-\tilde z,t-s)|R\langle|f(\cdot,\tilde{z},s)|\rangle' d\tilde z\,.\\  
 		 \end{eqnarray*}		 
                 Inserting the above estimates in the Duhamel's formula (\ref{Duhamel1}), we have                 
                 {\footnotesize{\begin{eqnarray*}
                 &&\int_0^{\infty}\int_0^t \langle|\nabla'\partial_z u_s(z,\cdot)|\rangle'dsdt\\
                 &\lesssim& \int_0^{\infty}\int_s^{\infty}\min\left\{\frac{1}{R},\frac{R}{(t-s)}\right\}\int_{\R}|\partial_z\Gamma_1(z-\tilde{z},t-s)|\langle|f(\cdot,\tilde{z},s)|\rangle'd\tilde z dtds\\
                 &\lesssim&\int_{\R}\left(\int_0^{\infty}\min\left\{\frac{1}{R},\frac{R}{\tau}\right\}|\partial_z\Gamma_1(z-\tilde{z},\tau)|d\tau\right)\int_0^{\infty}\langle|f(\cdot,\tilde{z},s)|\rangle'dsd\tilde z\\
		 &\lesssim& \sup_{\tilde{z}}\int_0^{\infty}\langle|f(\cdot,\tilde{z},s)|\rangle'ds\int_{\R}\int_0^{\infty}\min\left\{\frac{1}{R },\frac{R}{\tau}\right\}|\partial_z\Gamma_1(z-\tilde{z},\tau)|d\tau d\tilde z\\
		 &\stackrel{(\ref{x1})}{\lesssim}&\sup_{\tilde{z}}\int_0^{\infty}\langle|f(\cdot,\tilde{z},s)|\rangle'ds\int_0^{\infty}\min\left\{\frac{1}{R \tau^{\frac{1}{2}}},\frac{R}{\tau^{\frac{3}{2}}}\right\}d\tau\\
		 &\stackrel{(\ref{MIN})}{\lesssim}&\sup_{\tilde{z}}\int_0^{\infty}\langle|f(\cdot,\tilde{z},s)|\rangle'ds\,.
                 \end{eqnarray*}}}
                 Taking the supremum in $z$ we obtain the desired estimate. 
                 \newline           
		\underline{Argument for (\ref{A2.1})}
                 \newline
		 We recall that we want to show                
                   $$\int_0^{\infty}\int_0^{\infty}\langle|\nabla'\partial_zu_N|\rangle'\frac{dz}{z}dt\lesssim \int_{0}^{\infty}\int_0^{\infty}\langle|f|\rangle'\frac{dz}{z}dt ,$$    
		   where $u_N$ be the solution to the non-homogeneous heat equation with Neumann boundary conditions (\ref{A2}).
		   By the Duhamel's principle we have     
		    $$u_N(x',z,t)=\int_{s=0}^{t} u_{N_s}(x',z,t)ds, $$    
		    where  $u_{N_s}$ is solution to            
	          \begin{equation*}
                  \left\{\begin{array}{rclc}
		  (\partial_t-\Delta)u_{N_s}&=&0 \qquad & {\rm for } \quad  z>0, t>s\,,\\
		  \partial_zu_{N_s}&=&0  \qquad & {\rm for } \quad z=0, t>s\,,\\
		  u_{N_s}&=&f  \qquad & {\rm for } \quad z>0, t=s\,,\\
		  \end{array}\right.        
                  \end{equation*}                    
                  is the solution of problem (\ref{A1}).        
		 Extending this equation to the whole space by even reflection
		 \footnote{With abuse of notation we will denote with $u_{N_s}$ and $f$ their even reflection}, we are left to study the problem        
		 \begin{equation*}
                  \left\{\begin{array}{rclc}
		  (\partial_t-\Delta)u_{N_s}&=&0 \qquad & {\rm for } \quad z\in \R, t>s\,,\\
		  u_{N_s}&=&f \qquad & {\rm for } \quad  t=s\,,\\
		  \end{array}\right.        
                 \end{equation*}        
		the solution of which can be  represented via heat kernel as       
		\begin{align*}
		  u_{N_s}(x',z,t)&=\int_{\R}\Gamma(\cdot,z-\tilde z, t-s)\ast_{x'}f(\cdot,\tilde z,s)d\tilde z\\
		  &=\int_{0}^{\infty}\left[\Gamma(\cdot,\tilde z+z,t-s)+\Gamma (\cdot,\tilde z-z,t-s)\right]\ast_{x'}f(\cdot,\tilde z,s)d\tilde z\,.
		\end{align*}          		 
		 Applying $\nabla'\partial_z$ to the representation above
 		 \begin{eqnarray*}
                  &&\nabla'\partial_zu_{N_{s}}(x',z,t)\\
		 &=&\footnotesize{
 		 \begin{cases}
 		     \int_{0}^{\infty}\int_{\R^{d-1}}\Gamma_{d-1}(x'-\tilde{x'},t-s)\left(\partial_z\Gamma_1(\tilde z+z,t-s)-\partial_z\Gamma_1 ( \tilde z-z,t-s)\right)\nabla'f(\tilde{x'},\tilde z,s)d\tilde{x'}d\tilde z\,,\\
 		     \int_{0}^{\infty}\int_{\R^{d-1}}\nabla'^2\Gamma_{d-1}(x'-\tilde{x'},t-s)\left(\partial_z\Gamma_1(\tilde z+z,t-s)-\partial_z\Gamma_1 (\tilde z-z,t-s)\right)(-\Delta')^{-1}\nabla'f(\tilde{x'},\tilde z,s)d\tilde{x'}d\tilde z\,
 		    \end{cases}}
 		 \end{eqnarray*}    		
 		and averaging in the horizontal direction we obtain, on the one hand
%
                {\footnotesize{  
		\begin{eqnarray*}
                 &&\langle|\nabla'\partial_zu_{N_s}(\cdot,z,t)|\rangle'\\
                 &\lesssim&\int_{0}^{\infty}\langle|\Gamma_{d-1}(\cdot,t-s)|\rangle'|\partial_z\Gamma_1(\tilde z+z,t-s)-\partial_z\Gamma_1 (\tilde z-z,t-s)|\langle|\nabla'f(\cdot,\tilde z,s)|\rangle' d\tilde z\\
                 &\stackrel{(\ref{z0})\&(\ref{BAND2})}{\lesssim}&\frac{1}{R}\int_{0}^{\infty}|\partial_z\Gamma_1(\tilde z+z,t-s)-\partial_z\Gamma_1 (\tilde z-z,t-s)|\langle|f(\cdot,\tilde z,s)|\rangle' d\tilde z
 		 \end{eqnarray*}    }}
 		 and, on the other hand 
		  {\footnotesize{
 		\begin{eqnarray*}
                 &&\langle|\nabla'\partial_zu_{N_s}(\cdot,z,t)|\rangle'\\
 		 &\lesssim& \int_{0}^{\infty}\langle|\nabla'^2\Gamma_{d-1}(\cdot,t-s)|\rangle'|\partial_z\Gamma_1(\tilde z+z,t-s)-\partial_z\Gamma_1(\tilde z-z,t-s)|\langle|(-\Delta')^{-1}\nabla'f(\cdot,\tilde z,s)\rangle'd\tilde z\\
 		 &\stackrel{(\ref{z0})\&(\ref{BAND1})}{\lesssim}&\frac{R}{(t-s)} \int_{0}^{\infty}|\partial_z\Gamma_1(\tilde z+z,t-s)-\partial_z\Gamma_1(\tilde z-z,t-s)|\langle|f(\cdot,\tilde z,s)|\rangle'd\tilde z\,.
 		 \end{eqnarray*}  }}
		 Multiplying by the weight $\frac{1}{z}$ and integrating in $z\in(0,\infty)$ we get
		 \begin{equation*}
		\int_0^{\infty}\langle|\nabla'\partial_zu_{N_s}(\cdot,z,t)|\rangle'\frac{dz}{z}\lesssim \sup_{\tilde z}\int_0^{\infty}K_{t-s}(z,\tilde z) dz
		\begin{cases}
		\frac{1}{R}\int_{0}^{\infty}\langle|f(\cdot,\tilde z,s)|\rangle' \frac{d\tilde z}{\tilde z}\,,\\
		\frac{1}{(t-s)} R\int_{0}^{\infty}\langle |f(\cdot,\tilde z,s)|\rangle'\frac{d\tilde z}{\tilde z}\,,
		\end{cases}
		\end{equation*}    
		where we called $K_{t-s}(z,\tilde z)=\frac{\tilde z}{z}|\partial_z\Gamma_1(\tilde z-z,t-s)-\partial_z\Gamma_1( z+\tilde z,t-s)|$.
		\newline
		 Recalling 
		 $$\sup_{\tilde z}\int_0^{\infty}K_{t-s}(z,\tilde z) dz\stackrel{(\ref{EE1})}{\lesssim} \int_{\R}|\partial_z\Gamma_1(z,t-s)|dz+\sup_{z\in \R}(z^2|\partial^2_z\Gamma_1(z,t-s)|)$$
		 and observing that, in this case         
		 $$\int_{\R}|\partial_z \Gamma_1(z,t-s)|dz+\sup_{z\in \R}(z^2|\partial_z\Gamma_1(z,t-s)|)\stackrel{(\ref{x1})\&(\ref{y3})}{\lesssim}\frac{1}{(t-s)^{\frac{1}{2}}} ,$$
		 we can conclude that      		
		\begin{equation*}
		\int_0^{\infty}\langle|\nabla'\partial_zu_{N_s}(\cdot,t)|\rangle'\frac{dz}{z}\lesssim
		 \footnotesize{\begin{cases}
		\frac{1}{(t-s)^{\frac{1}{2}}}\frac{1}{R}\int_{0}^{\infty}\langle|f(\cdot,\tilde z,s)|\rangle' \frac{d\tilde z}{\tilde z}\\
		\frac{1}{(t-s)^{\frac{3}{2}}} R \int_{0}^{\infty}\langle |f(\cdot,\tilde z,s)|\rangle'\frac{d\tilde z}{\tilde z}.
		\end{cases}}
		\end{equation*}    
		Finally,  inserting (\ref{Duhamel1}) and integrating in time we have 
               \begin{eqnarray*}
               &&\int_0^{\infty}\int_0^{\infty}\langle|\nabla'\partial_z u_{N_s}(\cdot,z,t)|\rangle'\frac{dz}{z} dt \\
               &\stackrel{(\ref{Duhamel1})}{\lesssim}&\int_0^{\infty}\int_0^{\infty}\int_0^t\langle|\nabla'\partial_z u_{N_s}(\cdot,\tilde{z},t)|\rangle' \frac{dz}{z} ds dt\\
               &\lesssim&\int_s^{\infty} \int_0^{\infty}\min\{\frac{1}{R(t-s)^{\frac{1}{2}}},\frac{R}{(t-s)^{\frac{3}{2}}}\}\int_0^{\infty}\langle|f(\cdot,\tilde{z},s)|\rangle'\frac{d\tilde z}{\tilde z} ds dt\\
               &\stackrel{(\ref{Es1})\&(\ref{Es2})}{\lesssim}&\int_0^{\infty}\int_0^{\infty}\langle|f(\cdot,\tilde{z},s)|\rangle'\frac{d\tilde z}{\tilde z} ds\,.
               \end{eqnarray*}
		  \underline{Argument for (\ref{A3.1}):}
                   \newline
		  Recall that we need to prove 
		  $$\sup_z\int_0^{\infty}|\nabla'\partial_zu_C|dt\lesssim \langle|\nabla'\partial_zu|_{z=0}\rangle'\,.$$
                  By  equation (\ref{A3}), the even extension $\overline{u_C}$ satisfies                  		          
		  \begin{equation}\label{BO}
		   (\partial_t-\Delta)\overline{u_C}=-[\partial_z\overline{u_C}]\delta_{z=0}=-2\partial_z u_C\delta_{z=0}=-2\partial_z u|_{z=0}\delta_{z=0}
		  \end{equation}		          
		  and therefore we study the following problem on the whole space
                 \begin{equation}\label{A3b}
                  \left\{\begin{array}{rclc}
		   (\partial_t-\Delta)\overline{u_C}&=&-2\partial_z u|_{z=0}\delta \qquad & {\rm for } \quad z\in \R, t>0\,,\\
		  \overline{u_C}&=&0\qquad & {\rm for } \quad  t=0\,.\\
		  \end{array}\right.        
                 \end{equation}                  
                   By Duhamel's principle 
                   \begin{equation}\label{DU2}\overline{u_{C}}(x',z,t)=\int_{s=0}^{t} \overline{u_{C_s}}(x',z,t)ds ,\end{equation}
                  where $\overline{u_{C_s}}$ solves the initial value problem                       
	          \begin{equation}\label{A3bb}
                  \left\{\begin{array}{rclc}
		  (\partial_t-\Delta)\overline{u_{C_s}}&=&0 \qquad & {\rm for } \quad  z\in \R, t>s\,,\\
		  \overline{u_{C_s}}&=&-2\partial_z u|_{z=0}\delta  \qquad & {\rm for } \quad z\in \R, t=s\,.\\
		  \end{array}\right.        
                  \end{equation}                                    
		 The solution of problem (\ref{A3bb}) can be represented via the heat kernel as         
		  \begin{eqnarray*}
		   \overline{u_{C_s}}(x',z, t)&=&\int \Gamma(z-\tilde z,t-s)\ast_{x'}(-2\partial_z u|_{z=0}\delta)(\tilde z,s) d\tilde z,\\
		    &=&-2\Gamma(z,t-s)\ast_{x'}\partial_z u(z, s)|_{z=0}\,.
		  \end{eqnarray*}   
		  We apply $\nabla'\partial_z$ to the representation above
		 \begin{equation*}
                 \nabla'\partial_z\overline{u_{C_s}}(x',z,t)=
                  \int_{\R^{d-1}}-2\Gamma_{d-1}(x'-\tilde{x'},t-s)\partial_z\Gamma_1(z,t-s)\nabla'\partial_z u(\cdot,z, s)|_{z=0}d\tilde{x'}
		 \end{equation*}    		   
		 and then average in the horizontal direction, 		  
		 \begin{eqnarray*}
                 &&\langle|\nabla'\partial_z\overline{u_{C_s}}(x',z,t)|\rangle'\\
                 &\lesssim&\langle|\Gamma_{d-1}(x',t-s)|\rangle'|\partial_z\Gamma_1(z,t-s)|\langle|\nabla'\partial_z u(\cdot,z, s)|_{z=0}|\rangle'\\
                 &\stackrel{(\ref{z0})}{\lesssim}&  |\partial_z\Gamma_1(z,t-s)|\langle|\nabla'\partial_z u(\tilde{x'},z, s)|_{z=0}|\rangle'\,.
		 \end{eqnarray*}    
	         Inserting the previous estimate in the Duhamel formula \ref{DU2} and integrating in time we get
	         \begin{eqnarray}
	         &&\int_0^{\infty}\langle|\nabla'\partial_z\overline{u_{C}}(x',z,t)|\rangle'dt\notag\\
	         &\leq& \int_0^{\infty}\int_0^t\langle|\nabla'\partial_z\overline{u_{C_s}}(x',z,t)|\rangle'dsdt\notag\\
	         &\lesssim& \int_0^{\infty}\int_s^{\infty}|\partial_z\Gamma_1(z,t-s)|dt\langle|\nabla'\partial_z u(\tilde{x'},z, s)|_{z=0}|\rangle' ds\notag\\
	         &\stackrel{(\ref{y1})}{\lesssim}&  \int_0^{\infty}\langle|\nabla'\partial_z u(\tilde{x'},z, s)|_{z=0}|\rangle' ds\,.\label{here}
	          \end{eqnarray} 
	          The estimate (\ref{A3.1}) follows immediately after passing to the supremum in (\ref{here}).
	          \end{proof}	 
          
\section{Appendix}\label{App}

\subsection{Preliminaries}
We start this section by proving some elementary bounds and equivalences, coming directly from the 
definition of horizontal bandedness (\ref{BANDLIM2}). These will turn  to be crucial in the proof of the
main result.
     \begin{lemma}\label{Bandedness}\ \\
       \begin{enumerate}
         \item[a)]   If 
                     \begin{equation}\label{B1}
                     \F r(k',z,t)=0 \quad  {\rm{ unless }} \quad R|k'|\geq 4\,
                     \end{equation} then 
                     \begin{equation}\label{BAND1}
                     \langle|r(\cdot,z,t)| \rangle'\leq R\langle|\nabla' r(\cdot,z,t)|\rangle'\,.
                     \end{equation}

                    In particular
                    $$||r||_{(0,\infty)}\leq R ||\nabla' r||_{(0,\infty)}\,.$$
        
          \item[b)] If    
                    \begin{equation}\label{B2}
                    \F r(k',z,t)=0 \quad  {\rm{ unless }} \quad R|k'|\leq 1\,
                    \end{equation} then 
                    \begin{equation}\label{BAND2}
                    \langle|\nabla'r(\cdot,z,t)| \rangle'\leq \frac{1}{R}\langle| r(\cdot,z,t)|\rangle'\,.
                    \end{equation}
                    In particular
                    $$||\nabla'r||_{(0,\infty)}\leq \frac{1}{R} ||r||_{(0,\infty)}\,.$$

         \item[c)]  If    
                    $$\F r(k',z,t)=0 \quad  {\rm{ unless }} \quad 1\leq R|k'|\leq 4 $$ then
                    \begin{equation}\label{P}
                    ||\nabla'(-\Delta')^{-\frac{1}{2}}r||_{(0,\infty)}\sim ||r||_{(0,\infty)}\,,
                    \end{equation}
                    and 
                   \begin{equation}\label{Q}
                   ||(-\Delta')^{\frac{1}{2}}r||_{(0,\infty)}\sim||\nabla' r||_{(0,\infty)}\,.
                   \end{equation}        
        \end{enumerate}
       \end{lemma}
                  
      \begin{remark}
       All the results stated in Lemma \ref{Bandedness} are valid with the norm $||\cdot||_{(0,\infty)}$ replaced with $||\cdot||_{(0,1)}$.
      \end{remark}
      \begin{remark}
      Notice that from (\ref{P}) and (\ref{Q}), it follows
      \begin{equation}\label{R}
      ||\nabla'(-\Delta')^{-1}\nabla'\cdot r||_{(0,\infty)}\lesssim ||r||_{(0,\infty)}\,.
      \end{equation}  
      \end{remark}

\begin{proof}\ \\

  \begin{enumerate}
    \item[a)] By rescaling we may assume $R=1$.
 
	      Let $\phi\in\ES (\R^{d-1})$ be a Schwartz function such that 

	      $$\F \phi(k')=\begin{cases}
			      0 & \mbox{ for }  |k'|\geq 1\\
			      1 & \mbox{ for } |k'|\leq 1
			     \end{cases}$$					
	      and such that $\int_{\R^{d-1}} \phi(x')dx'=1.$

	      We claim that, under assumption (\ref{B1}), there exists $\psi \in L^1(\R^{d-1})$ such that 
	      \begin{equation}\label{O}
		(\rm{Id}-\phi\ast')r= \psi\ast'\nabla r\,.
	      \end{equation}
	      Since $r=r-\phi\ast r$, if we assume (\ref{O}) the conclusion follows from Young's inequality
	      $$ \int_{\R^{d-1}}|r(x',z)|dx'\leq \int_{\R^{d-1}}|\psi(x')|dx'\int_{\R^{d-1}}|\nabla r(x',z)|dx'\,.$$
	      \underline{Argument for (\ref{O}):}
	      \newline
	      Using the assumptions on $\phi$ and performing suitable change of variables, we find
	      \begin{eqnarray*}
	      &&r(x',z)-\int \phi(x'-y')r(y',z) dy'\\
	                    &=&\int\phi(x'-y')(r(x',z)-r(y',z)) dy'\\
                           &=&\int_{\R^{d-1}}\phi(x'-y')\int_0^1 (x'-y')\nabla' r(tx'+(t-1)(x'-y'), z) dy'dt\\
                           &=&\int_0^1\int_{\R^{d-1}}\phi(\xi)\nabla'r(x'+(t-1)\xi,z)\cdot \xi d\xi dt\\
                           &=&\int_0^1\int_{\R^{d-1}}\phi\left(\frac{\hat y'-x'}{t}\right)\nabla r(\hat y',z)\cdot \frac{\hat y'-x'}{t} dt \frac{1}{t^{d-1}}d\hat y'\\
                           &=&\int_{\R^{d-1}}\nabla' r(\hat y',z)\cdot \left(\int_0^1\phi\left(\frac{\hat y'-x'}{t}\right)\frac{\hat y'-x'}{t^{d}} dt\right) d\hat y'\\
                           &=&\int_{\R^{d-1}}\nabla' r(\hat y',z)\psi\left(\frac{\hat y'-x'}{t}\right)d\hat y',
	     \end{eqnarray*}
	     where 
	    $$\psi(x')=\int_0^1\phi\left(\frac{-x'}{t}\right)\frac{x'}{t^{d}}dt\,.$$
	    We notice that $\psi\in L^1(\R^{d-1})$, in fact
	    $$\int_{\R^{d-1}}|\psi(x')|dx'\leq \int_0^1 \int_{\R^{d-1}}\left|\phi(x'/t)\frac{x'}{t^{d}}\right| dx' dt=\int_{\R^{d-1}}|\phi(\xi)\xi|d\xi\,.$$
  
  \item[b)]
	  In Fourier space we have 
	  $$\F\nabla'r(k',z)=ik'\F r(k',z)=R^{-1}\F G(Rk')\F r(k',z)=R^{-1}\F G_R(k')\F r(k',z),$$
	  where $G$ is a Schwartz function and $G_{R}(x')=R^{-d}\F G(x'/R)$. Since 
	  $\int |G_R|dx'=\int |G| dx'$ is independent of $R$, we may conclude by Young
	  $$\int |\nabla' r| dx'\leq \frac{1}{R}\int |G_R| dx'\int |r| dx'\lesssim \frac{1}{R}\int |r| dx'\,. $$
	  
  \end{enumerate}

 \end{proof}

       
Here we prove an elementary estimate that will be applied in the argument for (\ref{A1.1}) and (\ref{A2.1}), Lemma \ref{lemma3}
\begin{lemma}\label{Lemma5}\ \\
Let $K=K(z)$ be a real function and define
 $$\overline{K}(z,\tilde z)=\frac{\tilde z}{z}|K(\tilde z- z)-K( z+\tilde z)|\,.$$
Then
\begin{equation}\label{EE1}
\sup_{\tilde z}\int_0^{\infty}\overline{K}(z,\tilde z) dz\lesssim \int_{\R}|K(z)|dz+\sup_{z\in \R}(z^2|\partial_zK(z)|)\,.
\end{equation}
 \end{lemma}
 
 \begin{proof}
Let us distinguish two regions:  $\frac{1}{2}\left|\frac{\tilde{z}}{z}\right|<1$ and $\frac{1}{2}\left|\frac{\tilde{z}}{z}\right|>1$. 
\newline
For $|z|\geq \frac{1}{2}|\tilde z|$ we have
\begin{eqnarray*}
&&\sup_{\tilde z}\int_{|z|\geq \frac{1}{2}|\tilde z|}|\overline{K}(z,\tilde z)|dz\\
&\leq &\max_{\tilde z}\int_{|z|\geq \frac{1}{2}|\tilde z|}|K(\tilde z-z)-K(z+\tilde z)|dz\lesssim \int|K(z)|dz\,.\\
\end{eqnarray*}
While for the region $|z|\leq \frac{1}{2}|\tilde z|$ we have,
\begin{eqnarray*}
&&\max_{\tilde z}|\tilde z|\int_{|z|\leq\frac{1}{2}|\tilde z|}\frac{1}{|z|}|K(\tilde z-z)-K(z+\tilde z)|dz\\
&=&\max_{\tilde z}|\tilde z|\int_{|z|\leq\frac{1}{2}|\tilde z|}\frac{1}{|z|}\left|\int_{-1}^1K'(\tilde z+t z) zdt\right|dz\\
&\leq&\max_{\tilde z}|\tilde z|\int_{-1}^1\frac{1}{t}\int_{|z|\leq\frac{t}{2}|\tilde z|}|K'(\tilde z+ z) |dzdt\\
&\stackrel{\frac{1}{2}|\tilde z|\leq |\tilde z+z|}{\leq}& \max_{\tilde z}\int_{-1}^1\frac{1}{t}\int_{|z|\leq\frac{t}{2}|\tilde z|}2|\tilde z+z||K'(\tilde z+ z)|dt dz\\
&\leq&\max_{\tilde z}\int_{-1}^1\frac{2}{t}\max_{|z|\leq\frac{t}{2}|\tilde z|}\left\{|\tilde z+z||K'(\tilde z+ z)|\right\}\left(\int_{|z|\leq\frac{t}{2}|\tilde z|} dz\right)dt\\
&=&\max_{\tilde z}\int_{-1}^1\frac{1}{t}\max_{|z|\leq\frac{t}{2}|\tilde z|}\{|\tilde z+z||K'(\tilde z+ z)|\} t|\tilde z|dt\\
&=&2\max_{\tilde z}|\tilde z|\max_{|z|\leq\frac{t}{2}|\tilde z|}\{|\tilde z+z||K'(\tilde z+ z)|\}\\
&\stackrel{\frac{1}{2}|\tilde z|\leq |\tilde z+z|}{\leq}&4\max_{\tilde z}\max_{|z|\leq\frac{t}{2}|\tilde z|}\{| z+\tilde z|^2|K'(\tilde z+ z)|\}\,.\\ 
\end{eqnarray*}
In conclusion we have 
\[\max_{ z}\int|\bar K(z,\tilde z)|dz\lesssim\int|K(z)|dz+\max_{ z}|z|^2|K'(z)| \,.\]
\end{proof}

\subsection{Heat kernel: elementary estimates}
In this section we recall the definition of the heat kernel and some 
properties and estimates that we will use throughout the paper.
\newline
The function $\Gamma:\R^{d}\times \R\rightarrow \R$ is defined as
$$\Gamma(x,t)=\frac{1}{t^{d/2}}\exp\left(-\frac{|x|^2}{4t}\right)$$
and we can rewrite it as 
$$\Gamma(x,t)=\Gamma_1(z,t)\Gamma_{d-1}(x',t) \qquad x'\in \R^{d-1}, z\in \R,$$
 where $$\Gamma_1(z,t)=\frac{1}{t^{1/2}}\exp\left(-\frac{z^2}{4t}\right)$$ and 
 $$\Gamma_{d-1}(z,t)=\frac{1}{t^{(d-1)/2}}\exp\left(-\frac{|x'|^2}{4t}\right)\,.$$ 
Here we list the bounds on the derivatives of $\Gamma$ that are used in Section \ref{App}, Lemma \ref{lemma3}: 
\begin{enumerate}
  \item  \begin{equation}\label{z0}
         \langle|(\nabla')^n\Gamma_{d-1}|\rangle'\approx \frac{1}{t^{\frac{n}{2}}}\,.
         \end{equation}
 \item  \begin{equation}\label{x1}
        \int_{\R}|\partial_z^{n}\Gamma_1|dz\lesssim \frac{1}{t^{\frac{n}{2}}}\,.
        \end{equation}
\item   \begin{equation}\label{y1}
        \int_{0}^{\infty}|\partial_z\Gamma_{1}(z,t)|dt=\int_0^{\infty}\left|\frac{1}{\hat{t}^{3/2}}\exp{\left(-\frac{1}{4\hat{t}}\right)}\right|d\tilde t\lesssim 1\,,
        \end{equation}
        where we have used the change of variable $\hat{t}=\frac{t}{z^2}$.
\item   \begin{equation}\label{y2}
        \sup_{z\in \R}\left(z|\partial_z\Gamma_1(z,t)|\right)=\sup_{\xi}\left|\frac{1}{t^{\frac{1}{2}}}\xi^{2}\exp^{-\xi^2}\right|\lesssim \frac{1}{t^{\frac{1}{2}}}\,,
        \end{equation}
        where we have used the change of variable $\xi=\frac{z}{t^{\frac{1}{2}}}$\,.
\item   \begin{equation}\label{y3}
        \sup_{z\in \R}\left(z^2|\partial_z\Gamma_1(z,t)|\right)=\sup_{\xi}\left|\xi^{3}\exp^{-\xi^2}\right|\lesssim 1\,,
        \end{equation}
       where we have used the change of variable $\xi=\frac{z}{t^{\frac{1}{2}}}$\,.
\end{enumerate}

\section{Notations}\label{notations}

\underline{The $(d-1)-$dimensional torus:}
\vspace{0.1cm}\newline
We denote with $[0,L)^{d-1}$  the $(d-1)-$dimensional torus of lateral size $L$.
\newline
\underline{The spatial vector:}

$$x=(x',z)\in [0,L)^{d-1}\times \R \,.$$

\underline{The horizontal average:}

\begin{equation*}
   \langle\cdot\rangle'=\frac{1}{L^{d-1}}\int_{[0,L)^{d-1}} \;\cdot\;\;\; dx'\,.
\end{equation*}
\underline{Long-time and horizontal average:}
\begin{equation}\label{LTaHA}
  \langle\cdot\rangle= \limsup_{t_0\rightarrow \infty}\frac{1}{t_0}\int_0^{t_0}\langle \;\cdot \;\rangle'  dt\,.
\end{equation}
\underline{Convolution in the horizontal direction:}

\begin{equation*}
 f\ast_{x'}g(x')=\int_{[0,L)^{d-1}}f(x'-\widetilde{x'})g(\widetilde{x'})d\widetilde{x'}\;.
\end{equation*}
\underline{Convolution in the whole space:}

\begin{equation*}
 f\ast g(x)=\int_{\R}\int_{[0,L)^{d-1}}f(x'-\widetilde{x'},z-\tilde{z})g(\widetilde{x'},\tilde{z})d\widetilde{x'}d\tilde{z}\,.
\end{equation*}

\underline{Horizontal Fourier transform:}

$$\mathcal{F'}f(k',z,t)=\frac{1}{L^{d-1}}\int e^{-ik'\cdot x'}f(x',z,t)dx'\,.$$
where $k'$ is the conjugate  variable of $x'$.

\underline{Horizontally band-limited function:}
\newline
A function $g=g(x',z,t)$ is called {\it horizontally  band-limited} with bandwidth $R$ if it satisfies 
the {\it bandedness assumption }
\begin{equation}\label{BANDLIM2}
 \F g(k',z,t)=0 \mbox{ unless } 1\leq R|k'|\leq 4 \mbox{ where } R<R_0.
\end{equation}
\underline{Interpolation norms:}
$$||f||_{(0,1)}=||f||_{R;(0,1)}=\inf_{f=f_1+f_2}\left\{\left\langle\sup_{z\in(0,1)}|f_1|\right\rangle+\left\langle\int_{(0,1)}|f_2|\frac{dz}{z(1-z)}\right\rangle\right\}\,,$$
$$||f||_{(0,\infty)}=||f||_{R;(0,\infty)}=\inf_{f=f_1+f_2}\left\{\left\langle\sup_{z\in(0,\infty)}|f_1|\right\rangle+\left\langle\int_{(0,\infty)}|f_2|\frac{dz}{z}\right\rangle\right\}\,,$$
$$||f||_{(-\infty,1)}=||f||_{R;(-\infty,1)}=\inf_{f=f_1+f_2}\left\{\left\langle\sup_{z\in(-\infty,1)}|f_1|\right\rangle+\left\langle\int_{(-\infty,1)}|f_2|\frac{dz}{1-z}\right\rangle\right\}\,.$$
where $f_0, f_1$ satisfy the bandedness assumption (\ref{BANDLIM2}).

\vspace{0.3cm}
Throughout the paper we will denote with $\lesssim$ the inequality up to universal constants.

\section*{Acknowledgement}
C.N. was supported by IMPRS of MPI MIS (Leipzig). A.C. was partially supported by Whittaker Research Fellowship.



\bibliographystyle{unsrt}
\bibliography{Biblio}

\end{document}